\documentclass[11pt]{amsart}
\usepackage{amsmath}
\usepackage{amsfonts,amssymb,amsthm,amstext,amscd,boxedminipage}  

\usepackage{graphicx}
\usepackage {psfrag} 
\usepackage{color}
\usepackage{tikz}
\usetikzlibrary{shapes,snakes}
 
 \numberwithin{equation}{section}

\DeclareMathSymbol{\minus} {\mathord}{operators}{"2D} %

\theoremstyle{plain}
\newtheorem{maintheorem}{Theorem}
\newtheorem{theorem}{Theorem}[section]
\newtheorem{lem}[theorem]{Lemma}
\newtheorem{cor}[theorem]{Corollary}
\newtheorem{proposition}[theorem]{Proposition}

\theoremstyle{definition}
\newtheorem{df}[theorem]{Definition}
\newtheorem{remark}[theorem]{Remark}
\newtheorem{example}[theorem]{Example}

\def \R {\mathbb{R}}
\def \C {\mathbb{C}}
\def \H {\mathbb{H}}
\def \Z {\mathbb{Z}}
\def \D {\mathbb{D}}
\def \M {\mathcal M}
\def \CN {\mathbb N}
\def \CO {\mathcal O}
\def \CR {\mathcal R}
\def \CS {\mathcal S}
\def \CW {\mathcal W}
\def \rk {{\mbox rank}}

\begin{document}

\title[Head and tail of the colored Jones polynomial]{Rogers-Ramanujan type identities and the head and tail of the colored Jones polynomial\\
}
\date{\today}

\author[C. Armond]{Cody Armond}
\address{Department of Mathematics, Louisiana State University,
Baton Rouge, LA 70803, USA}
\email{carmond@math.lsu.edu}

\author[O. T. Dasbach]{Oliver T. Dasbach}
\address{Department of Mathematics, Louisiana State University,
Baton Rouge, LA 70803, USA}
\email{kasten@math.lsu.edu}
\thanks {The second author was supported in part by NSF grants DMS-0806539 and DMS-0456275 (FRG). The first author was partially supported as a graduate student by NSF VIGRE grant DMS 0739382.}

\begin{abstract}
We study the head and tail of the colored Jones polynomial while  focusing mainly on alternating links. Various ways to compute the colored Jones polynomial for a given link give rise to combinatorial identities for those power series. We further show that the head and tail functions only depend on the reduced checkerboard graphs of the knot diagram. Moreover the class of head and tail functions of prime alternating links forms a monoid. 
\end {abstract}
 
\maketitle

\def\kbsm#1{\mathscr{S}_K(#1)}
\def\sgn#1;#2{\mathbb{S}_{#1,#2}} 
\def\mcg#1;#2{\Gamma_{#1,#2}} 
\def\fg#1;#2{\Pi_{#1,#2}}
\def\tb#1;#2{\mathscr{K}_{\frac{#1}{#2}}}
\def\periph{(\mathcal{\mu},\mathcal{\lambda})}
\def\ext#1{\mathscr{E}(\mathscr{#1})}
\def \qP #1 #2 #3 {(#1;#2)_{#3}}

\newcommand{\E}{\mathcal{E}}\def \R {\mathbb{R}}
\def \C {\mathbb{C}}
\def \H {\mathbb{H}}
\def \Z {\mathbb{Z}}
\def \D {\mathbb{D}}
\def \M {\mathcal M}
\def \CN {\mathbb N}
\def \CO {\mathcal O}
\def \CR {\mathcal R}
\def \CS {\mathcal S}
\def \CW {\mathcal W}
\def \rk {{\mbox rank}}
 \def \Vol {\mathrm{Vol}}
\def \tr {\mathrm{tr}}

\def\frametitle {}
\def\block {}

\section{Introduction}
One of the more fascinating objects in low dimensional topology is the colored Jones polynomial. To a knot $K$ one assigns a sequence of knot polynomials $J_{N,K}(q)$ in the Laurent ring $\Z[q,1/q]$ that is indexed by a natural number $N$. For each $N$ the colored Jones polynomial $J_{N,K}(q)$ is normalized to be $1$ for $K$ the unknot.
The ordinary Jones polynomial is $J_{2,K}(q)$. 

It was shown by Xiao-Song Lin and the second author \cite{DasbachLin:HeadAndTail} that for an alternating link $K$ the three leading coefficients of the colored Jones polynomial 
$J_N(K)$ are up to a common sign independent of $N$, for $N\geq 3$. Furthermore, the volume of the  hyperbolic link complement of an alternating link is bounded from above and below linearly in the absolute values of the second and the penultimate coefficient of the colored Jones polynomial.

It was conjectured \cite{DasbachLin:HeadAndTail} that more generally for an alternating link $K$ and every $N$ the $N$ leading coefficients of the colored Jones polynomial $J_{N,K}$ stabilize, i.e. they are - up to a sign - the $N$ leading coefficients of $J_{N+1,K}(q)$. For the right-handed trefoil and for the figure-8 knot those first $N$ coefficients are the first $N$ coefficients of the Dedekind $\eta$-function in $q=e^{2 \pi i \tau}$ (sometimes called Euler function). For the left-handed trefoil the first coefficient is $1$ followed by $N-1$ vanishing terms. A power series that describes the first $N$ coefficients of the colored Jones polynomial is called {\it tail}. The tail of $J_{N,K}(1/q)$ is called the {\it head}. For non-alternating links  a tail does not exist in general, e.g. for (negative) $(p,q)$-torus knots with $p, q>2$.
The proof of the conjecture for all alternating links was recently completed and will be posted in a separate paper by the first author \cite{Armond:HeadAndTailConjecture}. We do not make use of this result here. We learned that independently, L\^e and Garoufalidis also announced a proof of the conjecture for alternating knots \cite{Zagier:TalkAtWalterfest}. 

The main purpose of this paper is to understand the head and the tail of the colored Jones polynomial while focusing mainly on alternating knots. The paper has three parts. First, different ways to compute the colored Jones polynomial of knots and thus their head and tails lead to number theoretic identities of those power series. We give examples, and we will recover and reprove some well-known classical identities, e.g.

\begin{maintheorem}[Ramanujan]
$$\sum_{k=0}^{\infty} (-1)^k q^{(k^2+k)/2} = (q;q)_{\infty}^2 \sum_{k=0}^{\infty} \frac {q^k} {(q;q)_k^2},$$
\end{maintheorem}
where $(q;q)_k=\prod _{j=1}^{k-1} (1-q^j)$ are the $q$-Pochhammer symbols.

Second, we will show that the head and tail functions for alternating knots only depend on the reduced checkerboard graphs of the knots:

\begin{maintheorem}
Let $K_1$ and $K_2$ be two alternating links together with alternating diagrams with $A$-checkerboard graphs (respectively $B$-checkerboard graphs) $G_1$ and $G_2$. The reduced graphs $G_1'$ and $G_2'$ are obtained from $G_1$ and $G_2$ by replacing parallel edges by single edges. Then if $G_1'=G_2'$ the tails (respectively the heads) of the colored Jones polynomials of $K_1$ and $K_2$ coincide. 
\end{maintheorem}
More geometrically that means that if in an alternating link diagram $\pm k$ half twists are replaced by $\pm (k + l)$ half twist for some positive numbers $k$ and $l$ either the head or the tail of the colored Jones polynomial does not change. For example this explains why the tail of the colored Jones polynomial of the right-handed trefoil equals the tail of the colored Jones polynomial of the figure-8 knot.

Third, we will show a product structure on the head and tails for prime alternating links.

\begin{maintheorem}
 For an arbitrary pair of prime alternating links a new prime alternating link can be constructed whose tail of its colored Jones polynomial is the product of the tails of the colored Jones polynomial of the two links. In this sense the head and the tails of prime alternating links form a monoid.
\end{maintheorem} 

Lawrence and Zagier \cite{LawrenceZagier:ModularFormsQuantumInvariants} were the first to realize the occurrence of modular forms in the study of quantum invariants. Hikami \cite{Hikami:VCAsymptoticExpansion} understood that the series in the second Rogers-Ramanujan identity and in the Andrews-Gordon identities appear in the study of the colored Jones polynomial of torus knots. However, our point of view is very different in that the restriction to only the head and tail of the colored Jones polynomial allows in principle to prove combinatorial identities for heads of colored Jones polynomials of alternating knots in general, in a structured way.

Some examples of heads and tails for knots with small crossing number were recently computed by Garoufalidis, L\^e and Zagier \cite{Zagier:TalkAtWalterfest}. Many of the phenomena in their table can be explained with the methods developed here. Habiro \cite{Habiro:KnotsInPoland} computed examples of tails of the coefficient functions in the Habiro expansion of the colored Jones polynomial \cite{Habiro:Cyclotomic}. Those coefficient functions are called the reduced colored Jones polynomials.
   
The paper is structured as follows: In Section \ref{SettingTheScene} we give the necessary background from number theory and recall three combinatorial ways to compute the colored Jones polynomial. The first is a skein theoretical approach, the second a combinatorial interpretation of the first author \cite{Armond:Walks} of the quantum determinant formulation for the colored Jones polynomial of Huynh and L\^e \cite{VuLe:ColoredJonesDeterminant}. The third way is a combinatorial model based on a description of the colored Jones polynomial by Murakami \cite{Murakami:IntroToVolumeConjecture}. It generalizes ideas of Lin and Wang \cite{LinWang:RandomWalkColoredJones}. 
Section \ref{DependenceOnReducedCheckerboardGraph} will show that the head and tail functions only depend on the checkerboard graphs of the knot. Section \ref{ProductFormula} will give the product formula for heads and tails. 
In a final section we will briefly discuss (non-alternating) torus knots as examples of knots that are not alternating. 

\noindent
{\bf Acknowledgements: } Some of the initial ideas to this paper are based on work of the second author with Xiao-Song Lin. 
Moreover, we thank Pat Gilmer, Kazuo Habiro, Effie Kalfagianni, Matilde Lalin, Gregor Masbaum, Roland van der Veen and Don Zagier for helpful discussions at various occasions over the years.
The first author thanks Thang L\^e and Stavros Garoufalidis for stimulating discussions and the hospitality during a recent visit to Georgia Tech University.

Very important to the results in this paper were the search sites KnotInfo by Cha and Livingston \cite{ChaLivingston:KnotInfo} and
The On-Line Encyclopedia of Integer Sequences by Sloane \cite{Sloane:IntegerSequences}. Furthermore, Bar-Natan's Mathematica package KnotTheory \cite{BarNatan:KnotTheory} contains an implementation of an algorithm of Garoufalidis and L\^e to compute the colored Jones polynomial, which is feasible for knots with small crossing number and color $N<9$.

\section{Setting the scene}  \label{SettingTheScene}      

\subsection{Number theoretical background}

The $q$-Pochhammer symbol is defined as

$$\qP a q n :=\prod _{j=0}^{n-1} (1- a q^j)$$

and a fundamental identity is given by the Jacobi triple product identity (e.g. \cite{Wilf:JacobiTripleProduct}):
$$\sum_{m=-\infty}^{\infty} x^{m^2} y^{2m} = \prod_{n=1}^{\infty} (1-x^{2n}) (1+x^{2n-1} y^2) \left (1+\frac{x^{2n-1}} {y^2} \right ).$$

Two functions will be useful tools in our computations. The Ramanujan general theta function and the false theta function:

\begin{df} We will follow the notations in the literature:
\begin{enumerate}
\item The general (two variable) Ramanujan theta function (e.g. \cite{AndrewsBerndt:RamanujansLostNotebookI}):

\begin{eqnarray*}
f(a,b)&:=&\sum_{k=-\infty}^{\infty} a^{k(k+1)/2} b^{k (k-1)/2}\\
&=& \sum_{k=0}^{\infty} a^{k(k+1)/2} b^{k(k-1)/2} + \sum_{k=1}^{\infty} a^{k(k-1)/2} b^{k(k+1)/2}
\end{eqnarray*}

By the Jacobi triple product identity $f(a,b)$ can be expressed as:

$$f(a,b)=\qP {-a} {a b} {\infty}  \qP {-b} {a b} {\infty} \qP {a b} {a b} {\infty} .$$

In particular $f(-q^2,-q)=(q;q)_{\infty}$. 

Note that $f(a,b)=f(b,a).$

\item The false theta function (e.g. \cite{McLaughlin:RogersRamanujanSlater}):
$$\Psi(a,b):=\sum_{k=0}^{\infty} a^{k(k+1)/2} b^{k(k-1)/2} - \sum_{k=1}^{\infty} a^{k(k-1)/2} b^{k(k+1)/2}$$

Note that $\Psi(a,b)-2 = -\Psi(b,a).$

\end{enumerate}
\end{df}

\subsubsection{Rogers-Ramanujan type identities; and page 200 in Ramanujan's lost notebook}

\begin{enumerate}
\item In terms of the Ramanujan theta function the celebrated (second) Rogers-Ramanujan identity states:

$$
f(-q^4,-q) = (q;q)_{\infty} \sum_{k=0}^{\infty} \frac{q^{k^2+k}}{(q;q)_k}        
$$       

The Andrews-Gordon identities are generalizations of the second Rogers-Ramanujan identity:

\begin{equation} \label{Andrews-Gordon}
f(-q^{2k},-q)= (q;q)_{\infty} \sum_{n_1,\dots,n_{k-1}\geq 0} \frac {q^{N_1^2+\dots+N_{k-1}^2+N_1+\dots+N_{k-1}}} {(q;q)_{n_1} \cdots (q;q)_{n_{k-1}}}
\end{equation}
with $N_j$ defined as
$$N_j=n_j+\dots+n_{k-1}.$$

\item A corresponding identity for the false theta function is:
\begin{eqnarray*}
\Psi(q^3,q)&=& \sum_{k=0}^{\infty} (-1)^k q^{(k^2+k)/2}\\
&=& (q;q)_{\infty} \sum_{k=0}^{\infty} \frac{q^{k^2+k}}{(q;q)^2_k} \hspace{1cm} \mbox{ (Ramanujan's notebook, Part III, Entry 9 \cite{RamanujanNotebooks:PartIII}})\\
&=& (q;q)^2_{\infty} \sum_{k=0}^{\infty} \frac{q^k} {(q;q)_k^2} \hspace{1cm} \mbox{ (Ramanujan's lost notebook; page 200 \cite{AndrewsBerndt:RamanujansLostNotebookI})}
\end{eqnarray*}
\end{enumerate}

The equality between the series in line 1 and line 3 will be proven in this paper.

\subsection{How to compute the colored Jones polynomial?}

We briefly recall three ways to compute the colored Jones polynomial of a link that are going to be important to us. For links we assume that all
components are colored with the same color. For a link $L$ we denote by $J_{N,L}(q)$ the reduced colored Jones polynomial of the link, i.e. for an unknot $U$ the colored Jones polynomial is $J_{N,U}(q)=1.$ Furthermore, for $N=2$ the colored Jones polynomial $J_{2,K}(q)$ equals the ordinary Jones polynomial of $K$.

The first approach is the skein theoretical approach. It gives with $A^{-4}=q$ a graphical calculus to compute the (unreduced) colored Jones polynomial 
$J_{N+1,L}(q)$ up to a power of $q$.

\subsection{Skein theory} \label{SectionSkeinTheory}

For details see e.g. \cite{Lickorish:KnotTheoryBook, MasbaumVogel:3valentGraphs}

By convention an $n$ next to a component of a link indicates that the component is replaced by $n$ parallel ones. The Jones-Wenzl idempotent is indicated by
a box on a component. It can be defined as follows:

With $$\Delta_{n}:=(-1)^{n} \frac {A^{2(n+1)}-A^{-2 (n+1)}}{A^{2}-A^{-2}}$$ and
$\Delta_{n}!:= \Delta_{n} \Delta_{n-1} \dots \Delta_{1}$ the Jones-Wenzl idempotent satisfies

$$\begin{tikzpicture}[baseline=0.8cm]
\draw(0,0)--(0,1) 
node[rectangle, draw, ultra thick, fill=white]{ } --(0,2) node[right]{\scriptsize{n+1}};
\end{tikzpicture} =
\begin{tikzpicture}[baseline=0.8cm]
\draw(0,0)--(0,1) 
node[rectangle, draw, ultra thick, fill=white]{ } --(0,2) node[right]{\scriptsize{n}};
\draw(0.7,0)--(0.7,2) node[right]{\scriptsize{1}};
\end{tikzpicture} -
\left( \frac {\Delta_{n-1}} {\Delta_{n}} \right ) \, \, 
\begin{tikzpicture}[baseline=0.8cm]
\draw (0,0) node[right]{\scriptsize{n}}--(0,1) node[left]{\scriptsize{n-1}}
--(0,2) node[right]{\scriptsize{n}};
\draw[ultra thick, fill=white] (-0.2 ,0.5) rectangle (0.4, 0.65); 
\draw (0.2, 0.65) arc (180:0:0.25);
\draw (0.7,0) node[right]{1}--(0.7, 0.65);
\draw[ultra thick, fill=white] (-0.2 ,1.5) rectangle (0.4, 1.35); 
\draw (0.2, 1.35) arc (-180:0:0.25);
\draw (0.7,2) node [right]{\scriptsize{1}} --(0.7, 1.35);
\end{tikzpicture}
$$

with the properties

$$\begin{tikzpicture}[baseline=0.8cm]
\draw(0,0)--(0,1) 
node[rectangle, draw, ultra thick, fill=white]{ } --(0,2) node[right]{\scriptsize{i+j}};
\end{tikzpicture} =
\begin{tikzpicture}[baseline=0.8cm]
\draw(0,0) node[right]{\scriptsize{i+j}}--(0,0.6);
\draw[ultra thick, fill=white] (-.4,0.6) rectangle (.4, 0.8);
\draw (0.2,0.8)--
(0.2,1.4)  node[rectangle, draw, ultra thick, fill=white]{ } --
(0.2,2)node[right]{\scriptsize{j}};
\draw (-0.2,0.8)--(-0.2,2) node[right]{\scriptsize{i}};
\end{tikzpicture},
\, \qquad \qquad \qquad
\begin{tikzpicture}[baseline=0.8cm]
\draw(0,0)--(0,1) 
node[rectangle, draw, ultra thick, fill=white]{ } --(0,2) node[right]{\scriptsize{1}};
\end{tikzpicture} = \, \, 
\begin{tikzpicture}[baseline=0.8cm]
\draw(0,0)--(0,2) 
node[right]{\scriptsize{1}};
\end{tikzpicture}$$ and
$$\begin{tikzpicture}[baseline=0cm]
\draw (0,0) circle(0.8);
\draw (-0.8,0)  node[left] {\scriptsize{n}};
\draw (0.75,0)  node[rectangle, draw, ultra thick, fill=white]{ } ;
\end{tikzpicture} = \Delta_{{n}}.
$$

A triple $(a,b,c)$ admissible if 

\begin{enumerate}
\item $a+b+c$ is even
\item $|a -b | \leq c \leq a+b.$
\end{enumerate}

For an admissible triple $(a,b,c)$ a trivalent vertex is defined by:

$$
\begin{tikzpicture}[baseline=2.8cm]
\draw (1,1.6) node[right]{\scriptsize{c}}-- (1,3);
\draw (1,3)--(0,4) node[left] {\scriptsize{a}};
\draw (1,3)--(2,4) node[right] {\scriptsize{b}};
    \fill (1,3) circle (0.07);
\end{tikzpicture}
=
\begin{tikzpicture}[baseline=2.8cm]
\draw (1.7,3.5)--(1.7,4) node[right] {\scriptsize{b}};
\draw (0.3,3.5)--(0.3,4) node[left] {\scriptsize{a}};
\draw[ultra thick, fill=white] (1.4,3.35) rectangle (2.0, 3.5);
\draw[ultra thick, fill=white] (0,3.35) rectangle (0.6, 3.5);
\draw (0.45,3.35) .. controls (0.8,2.9) and (1.2,2.9) .. (1.55, 3.35);
\draw (1, 3.2) node {\scriptsize{k}};
\draw (1,1.6) node[right]{\scriptsize{c}}-- (1,2.1);
\draw[ultra thick, fill=white] (0.7,2.1) rectangle (1.3, 2.25);
\draw (0.85,2.25)-- (0.15, 3.35)  ;
\draw (0.45, 2.8) node[left] {\scriptsize{i}};
\draw(1.15,2.25)--(1.85, 3.35);
\draw (1.5,2.8) node[right] {\scriptsize{j}};
\end{tikzpicture} 
$$

where $i,j$ and $k$ are defined by  $i+k=a, j+k=b$ and $i+j=c$.
In a trivalent diagram all edges are implied to be equipped with the Jones-Wenzl idempotent. 

Fusion is given by

$$\begin{tikzpicture}[baseline=0.8cm]
 \draw (0.7,0) .. controls (0.2 ,0.6) and (0.2,1.4) .. (0.7,2) node[right]{\scriptsize{b}};
 \draw (0.35,1) node[shape=rectangle, draw, ultra thick, fill=white]{ };
\draw (-0.7,0).. controls (-.2,0.6) and (-0.2,1.4)  .. (-0.7,2) node[left]{\scriptsize{a}};
 \draw (-0.35,1) node[rectangle, draw, ultra thick, fill=white]{ };
\end{tikzpicture}
=  \sum_{c} \frac {\Delta_c} {\theta(a,b,c)} 
\begin{tikzpicture}[baseline=0.8cm]
 \draw (1,0) node[right]{\scriptsize{b}} --(0.5,0.5);
 \draw (0,0) node[left]{\scriptsize{a}} --(0.5,0.5);
 \draw (0.5,0.5)--(0.5,1) node[right]{\scriptsize{c}}-- (0.5,1.5);
\draw (0.5,1.5)--(0,2) node[left] {\scriptsize{a}};
\draw (0.5,1.5)--(1,2) node[right] {\scriptsize{b}};
   \fill (0.5,0.5) circle (0.07);
   \fill (0.5,1.5) circle (0.07);
\end{tikzpicture} 
$$

where the sum is over all $c$ such that $(a,b,c)$ is admissible.

\bigskip To define $\theta(a,b,c)$ let $a,b$ and $c$ related as above and  $x, y$ and $z$ be defined by $a=y+z, b=z+x$ and $c=x+y$ then

$$\theta(a,b,c):= \, \begin{tikzpicture}[baseline=0.8cm, rounded corners=2mm]
\draw (0,1) ellipse (0.6 and 1) ;
\draw (0.7,0.96) node[right]{\scriptsize{c}} ;
\draw (-0.7,0.96) node[left]{\scriptsize{a}};
\draw (0,0)--(0,1) node[right]{\scriptsize{b}}--(0,2);
   \fill (0,0) circle (0.07);
   \fill (0,2) circle (0.07);
\end{tikzpicture}$$

and one can show that

$$\theta(a,b,c)= \frac { \Delta_{x+y+z}! \Delta_{x-1}! \Delta_{y-1}! \Delta_{z-1}! } {\Delta_{y+z-1}! \Delta_{z+x-1}! \Delta_{x+y-1}!}.$$

Furthermore one has:

$$\begin{tikzpicture}[baseline=.5cm, rounded corners=2mm]
    \draw(0,0)-- (-.5,.5) -- (-.1,0.9);
    \draw (.1,1.1)--(.5,1.5) node[right]{\scriptsize{b}};
      \draw (0,0) -- (.5,.5)--(-.5,1.5) node[left]{\scriptsize{a}};
     \draw (0,0)--(0,-0.8) node[right]{\scriptsize{c}};
     \fill (0,0) circle (0.07);
\end{tikzpicture}  = (-1)^{\frac{a+b-c} 2} A^{a+b-c+\frac{a^2+b^2-c^2}{2}}
\begin{tikzpicture}[baseline=.5cm, rounded corners=2mm]
       \draw (0,0)--(0,-0.8) node[right]{\scriptsize{c}}; 
     \draw (0,0)--(0.5,1.5) node[right]{\scriptsize{b}};
     \draw (0,0)--(-0.5,1.5) node[left]{\scriptsize{a}};
      \fill (0,0) circle (0.07);
      \end{tikzpicture}$$

We are only interested in the list of coefficients of the colored Jones polynomial. In particular we consider polynomials up to powers of their variable. Up to a factor of $\pm A^s$ for some power $s$ that depends on the writhe of the link diagram the (unreduced) colored Jones polynomial $\tilde J_{N+1,L}(A)$ of a link $L$ can be defined as the value of the skein relation applied to the link were every component is decorated by an $N$ together with the Jones-Wenzl idempotent.  Recall that $A^{-4}=q$. To obtain the reduced colored Jones polynomial we have to divide $\tilde J_{N,L}(A)$ by its value on the unknot. Thus 
$$J_{N+1,L}(A):=\frac{\tilde J_{N+1,L}(A)} {\Delta_{N}}.$$
 
\subsubsection{Walks and the quantum determinant formulation of Huynh and L\^e}

We briefly recall the combinatorial walk model for computing the colored Jones polynomial developed in \cite{Armond:Walks} that is based on the quantum determinant formulation of Huynh and L\^e \cite{VuLe:ColoredJonesDeterminant}.For more details see \cite{Armond:Walks}.
Given a braid $\beta$ whose closure is the knot $K$, a walk along $\beta$ is defined as follows:

At each element of a subset $J$ of the lower ends of the strands of $\beta$, begin walking up the braid. Upon reaching a crossing, the strand walked along is the over strand, then there is a choice to continue along the over strand, or jump-down to the under strand and continue walking up the braid. If the strand is the under-strand of the crossing, then continue along that strand. The set of ending positions must match the set of starting positions, thus inducing a permutation $\pi$ of the set $J$. A walk is simple if it does not traverse any arc more than once.

Each walk is assigned a weight as follows: Beginning from the first element of $J$ follow the walk and record $a_{j,\pm}$, for a jump down at crossing $j$, $b_{j,\pm}$ for following the under-strand at crossing $j$, and $c_{j,\pm}$ for following the over-strand at crossing $j$. Here the $\pm$ indicates the sign of the crossing. Once the top of the braid is reached, begin again from the next element of $J$. The weight of the walk is $(-1)(-q)^{|J|+\text{inv}(\pi)}$ times the product of the previous letters with the order as explained above.

The letters $a_{j,\pm}$, $b_{j,\pm}$, and $c_{j,\pm}$ satisfy the relations:

\begin{tabular}{c c c}
$a_{j,+}b_{j,+}=b_{j,+}a_{j,+},$ & $a_{j,+}c_{j,+}=qc_{j,+}a_{j,+},$ & $b_{j,+}c_{j,+}=q^2c_{j,+}b_{j,+}$\\
$a_{j,-}b_{j,-}=q^2b_{j,-}a_{j,-},$ & $c_{j,-}a_{j,-}=qa_{j,-}c_{j,-},$ & $c_{j,-}b_{j,-}=q^2b_{j,-}c_{j,-}$\\
\end{tabular}

Also any letters with different subscripts commute.

The weights also have an evaluation:
$$\E_N(b_+^s c_+^r a_+^d)=q^{r(N-1-d)} \prod_{i=0}^{d-1}(1-q^{N-1-r-i})$$
$$\E_N(b_-^s c_-^r a_-^d)=q^{-r(N-1)}\prod_{i=0}^{d-1}(1-q^{r+i+1-N})$$
In a word with letters of different subscripts, the word is separated into pieces corresponding to each subscript and each is evaluated separately.

\begin{theorem} \cite{Armond:Walks}
The colored Jones polynomial of $K$ is 
$$J_{N,K} = q^{(N-1)(\omega(\beta)-m+1)/2}\sum_{n=0}^\infty \E_N(C^n)$$
where $\omega(\beta)$ is the writhe of $\beta$, $m$ is the number of strands in $\beta$, and $C$ is the sum of the weights of the simple walks along $\beta$ with $J \subset \{2,\ldots,m\}$.
\end {theorem}

\begin{remark}
This sum is finite due to the fact that any monomial in $C^n$ corresponding to a collection of walks with at least $N$ of the walks traversing any given point will evaluate to $0$.
\end{remark}

\begin{example}[The $(2,5)$-torus knot] \label{25torus}
The (positive) $(2,5)$-torus knot $T(2,5)$ can be expressed as the closure of the braid $\beta = (\sigma_4\sigma_3\sigma_2\sigma_1)^2$

\begin{figure}[htbp] 
  \centering
  \includegraphics[width=1.25in]{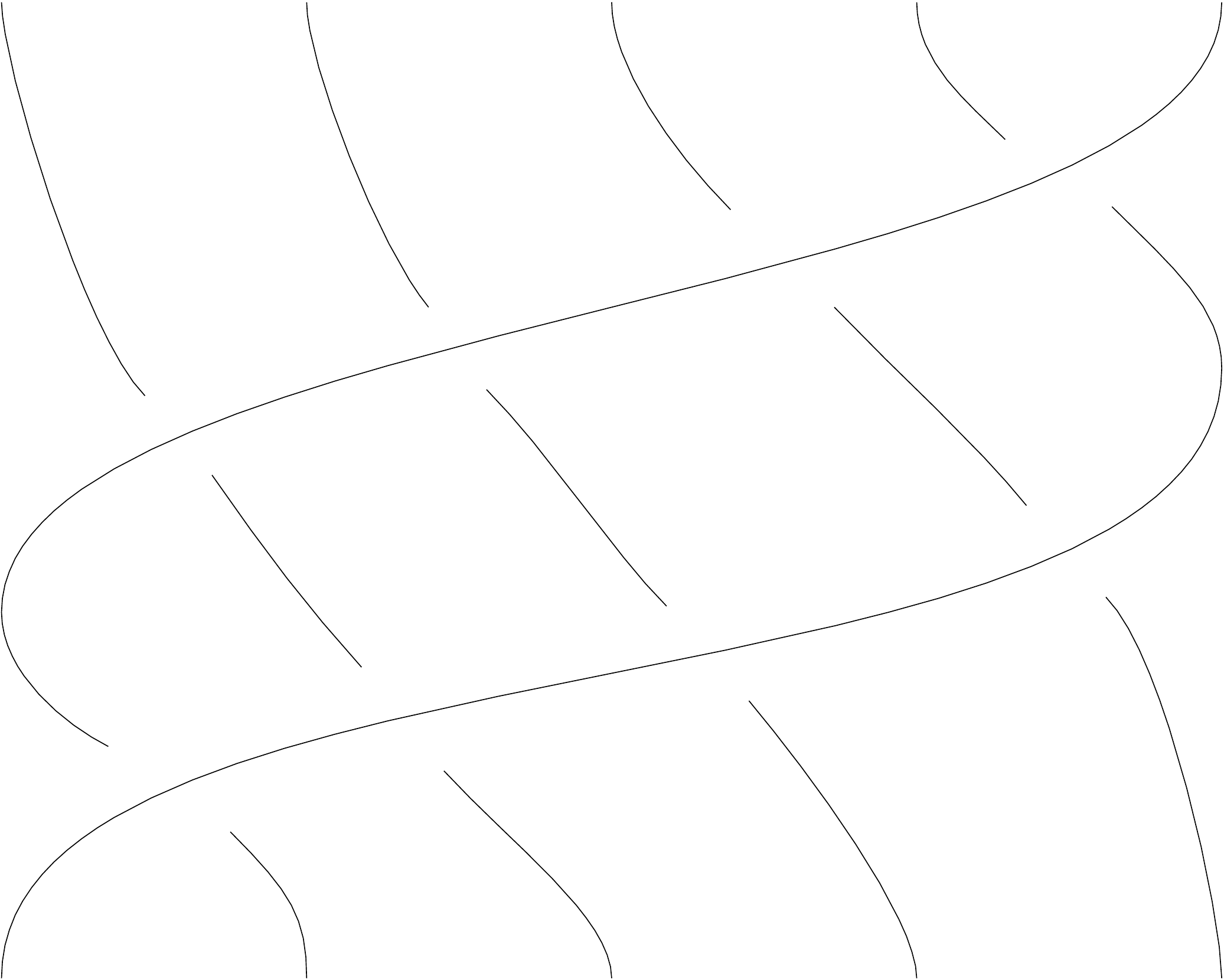} 
\end{figure}

The two simple walks along $\beta$ are the following:

\begin{center}
\begin{tabular}{c c}
\includegraphics[width=1.25in]{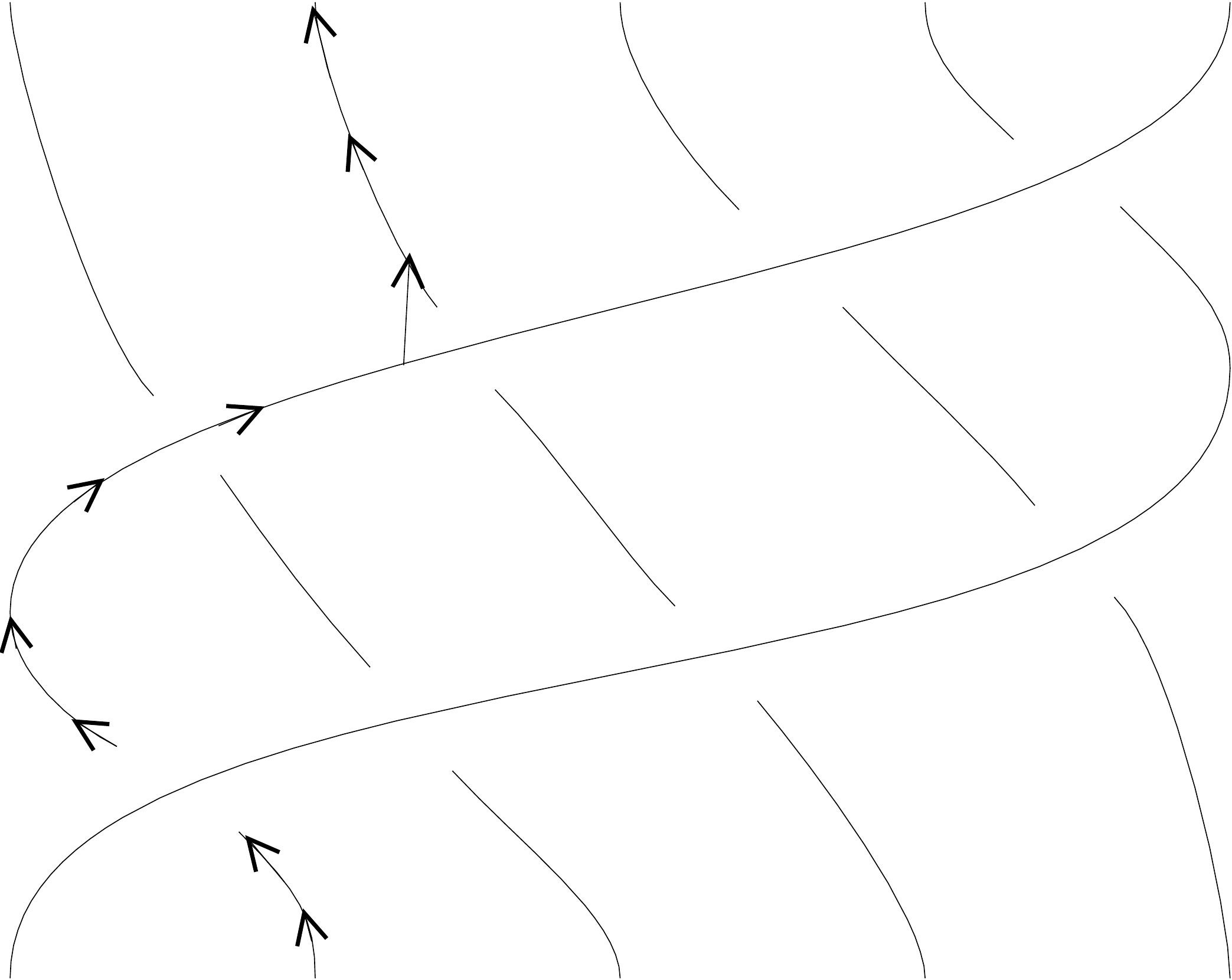}&\includegraphics[width=1.25in]{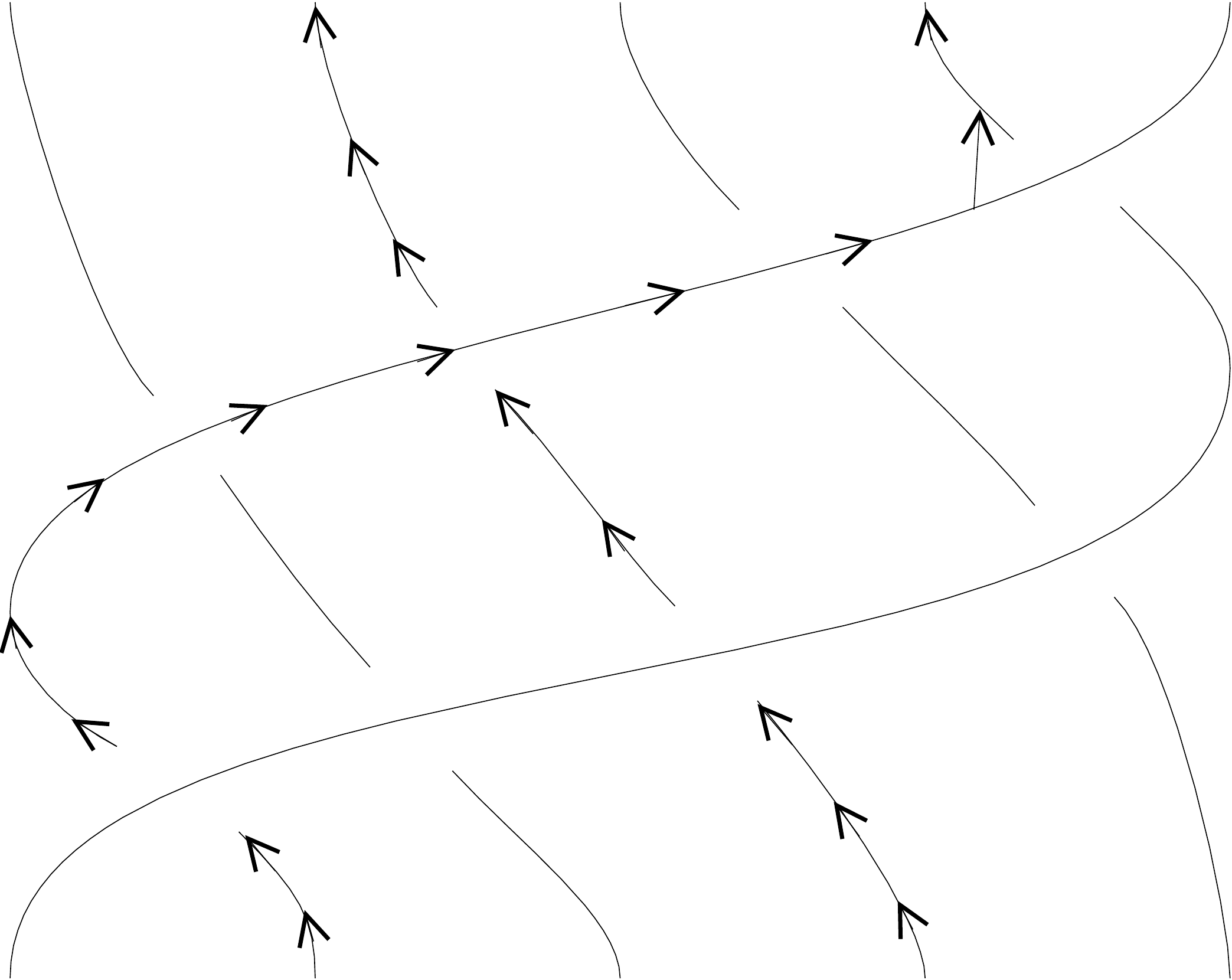}\\
   $W_1$ & $W_2$ 
\end{tabular}
\end{center}

with weights $W_1=qb_{8,+}c_{4,+}a_{3,+}$, and $W_2=q^3b_{8,+}c_{4,+}c_{3,+}c_{2,+}a_{1,+}b_{6,+}b_{3,+}$ with the relation $W_1W_2=qW_2W_1$. Thus for $K$ the $(2,5)$-torus knot the colored Jones polynomial $J_{N,K}(q)$ is, with $${n\choose k}_{q}:=\frac{(q;q)_n}{(q;q)_{n-k} (q;q)_k}$$ the $q$-binomials:

\begin{eqnarray*}
&&q^{2(N-1)}\sum_{n=0}^{N-1}\E_N((W_1+W_2)^n)\\
&=&q^{2(N-1)}\sum_{n=0}^{N-1}\sum_{k=0}^n  {n \choose k}_{q} \E_N(W_2^kW_1^{n-k})\\
&=&q^{2(N-1)}\sum_{n=0}^{N-1}\sum_{k=0}^n \left( \begin{matrix}n\\k\end{matrix}\right)_{q} q^{3k+(n-k)}\E_N(b_{8,+}^n b_{6,+}^{k} c_{4,+}^n (c_{3,+}b_{3,+})^k a_{3,+}^{n-k} c_{2,+} ^k a_{1,+}^k )\\
&=&q^{2(N-1)}\sum_{n=0}^{N-1}\sum_{k=0}^n \left( \begin{matrix}n\\k\end{matrix}\right)_{q} q^{2k+n-k(k+1)}\E_N(b_+^n) \E_N(b_+^k) \E_N(c_{+}^n) \E_N(b_{+}^k c_{+}^k a_{+}^{n-k}) \E_N(c_{+}^k) \E_N(a_{+}^k)\\
&=& q^{2(N-1)}\sum_{n=0}^{N-1} \sum_{k=0}^n \left( \begin{matrix}n\\k\end{matrix}\right)_{q} q^{nN+k(2N-1-n)} \prod_{i=0}^{n-1}(1-q^{N-1-i})
\end{eqnarray*}

\end{example}

\subsubsection{Walks and the Hitoshi Murakami's formulation of the colored Jones polynomial}

The concept of describing quantum link invariants via $R$-matrices goes back to Turaev \cite{Turaev:YangBaxter}.
In particular one can use $R$-matrices to compute the colored Jones polynomial.  The following discussion is a simplification of the description given by Hitoshi Murakami in \cite{Murakami:IntroToVolumeConjecture}. 
We will need to define the $R$-matrix, its inverse and another function $\mu$ (see also \cite{KM:cablingFormulaCoJP}):
\begin{eqnarray*}
R^{i~j}_{k~l} &=& \sum_{m=0}^{\text{Min}(N-1-i,j)} (-1)^m \delta_{l,i+m} \delta_{k,j-m}\frac{\{l\}!\{N-1-k\}!}{\{i\}!\{m\}!\{N-1-j\}!}\\
&& \times q^{-(i-(N-1)/2)(j-(N-1)/2)+m(i-j)/2+m(m+1)/4}
\end{eqnarray*}

\begin{eqnarray*}
(R^{-1})^{i~j}_{k~l} &=& \sum_{m=0}^{\text{Min}(N-1-j,i)}\delta_{l,i-m} \delta_{k,j+m}\frac{\{k\}!\{N-1-l\}!}{\{j\}!\{m\}!\{N-1-i\}!}\\
&& \times q^{(i-(N-1)/2)(j-(N-1)/2)+m(i-j)/2-m(m+1)/4}
\end{eqnarray*}

\begin{eqnarray*}
\mu_j &=& q^{-(2j-N+1)/2}
\end{eqnarray*}
where $\{m\} = q^{m/2} - q^{-m/2}$, and $\{m\}! = \prod_{i=1}^m \{i\}$. These together form an enhanced Yang-Baxter operator, but for our purposes we only need to think of them as functions of the integers $i,j,k,$ and $l$ where $0\leq i,j,k,l \leq N-1$.

To calculate the $N$-th colored Jones polynomial, take a braid $\beta$ whose closure is the knot $K$ and close each strand except for the left most one. Choose a label which is an integer between $0$ and $N-1$ for the top and bottom of the left-most strand which will be fixed throughout. A state is a choice of labels on each of the other arcs of this semi-closure. Assign to each crossing the polynomial $(R^{\epsilon_c})^{i~j}_{k~l}$ as in Figure \ref{crossR}, where $\epsilon_c$ is the sign of crossing $c$. Also assign to each arc which attaches the bottom of the braid to the top of the braid the polynomial $\mu_j$ where $j$ is the label of that arc. The weight of the state is the product of the $R$-matrix for each crossing times the product of $\mu_j$'s. Finally the colored Jones polynomial is $q^{-\omega(\beta)(N^2-1)/4}$ times the sum of the weights of each state.

\begin{figure}[h]
$$\begin{tikzpicture}[scale=1, baseline=0cm]
    \draw (-.5, -0.5)node[left]{\tiny{k}}-- (0,0) -- (.5,.5) node{\tiny{\, j}};
    \draw (.5,-0.5)node{\tiny{\, l}}  -- (0.1,-0.1);
    \draw (-0.1,0.1) -- (-.5,.5) node[left]{\tiny{i}};
   \end{tikzpicture}  \longrightarrow R^{i~j}_{k~l}
$$

$$\begin{tikzpicture}[scale=1, baseline=0cm]
    \draw (-.5, -0.5)node[left]{\tiny{k}}-- (-0.1,-0.1);
    \draw (.5,-0.5)node{\tiny{\, l}}  -- (0,0)-- (-.5,.5) node[left]{\tiny{i}};
    \draw (0.1,0.1) -- (.5,.5) node{\tiny{\, j}};
   \end{tikzpicture}  \longrightarrow (R^{-1})^{i~j}_{k~l}
$$
\caption{The $R$-matrix}
\label{crossR}
\end{figure}

The delta functions in the definition of the $R$-matrix tells us that many of the states will have weight $0$. Figure \ref{jump} gives the conditions for which a state will have a non-zero weight:

\begin{figure}[h]
$$\begin{tikzpicture}[scale=1, baseline=0cm]
    \draw (-.5, -0.5)node[left]{\tiny{k}}-- (0,0) -- (.5,.5) node{\tiny{\, j}};
    \draw (.5,-0.5)node{\tiny{\, l}}  -- (0.1,-0.1);
    \draw (-0.1,0.1) -- (-.5,.5) node[left]{\tiny{i}};
   \end{tikzpicture}  : i+j=k+l, l\geq i, k\leq j,
$$

$$\begin{tikzpicture}[scale=1, baseline=0 cm]
    \draw (-.5, -0.5)node[left]{\tiny{k}}-- (-0.1,-0.1);
    \draw (.5,-0.5)node{\tiny{\, l}}  -- (0,0)-- (-.5,.5) node[left]{\tiny{i}};
    \draw (0.1,0.1) -- (.5,.5) node{\tiny{\, j}};
   \end{tikzpicture}  : i+j=k+l, l\leq i , k\geq j.
$$
\caption{}
\label{jump}
\end{figure}

This tells us that the labels can be thought of in terms of walks. The label represents how many ``walkers'' are walking along the labeled arc, walking from top to bottom. The conditions given in Figure \ref{jump} tells us that some number of the walkers walking along the over strand can jump down onto the lower strand, but no strand can have more than $N-1$ walkers on it. The walk model for the Jones polynomial that Xiao-Song Lin and Zhenghan Wang give in \cite{LinWang:RandomWalkColoredJones} coincides with this model of the colored Jones polynomial for $N=2$.

\begin{example}[The $(2,-4)$ torus link]
\label{toruslink}
Consider the braid $\sigma_1^{-4}$. The closure of this braid is the $(2,-4)$ torus link. To calculate the colored Jones polynomial we must first choose a label for the top left strand. We can choose any integer between $0$ and $N-1$, but choosing $0$ will simplify the calculations, so we will choose $0$.

Next we must find all of the non-zero states. There an no restrictions on the top right arc, so we shall label it $j$. This forces the arcs below the first crossing to be labeled $j$ and $0$:

$$\begin{tikzpicture}[scale=0.8, baseline=0.5cm, rounded corners=2mm]
    \draw (-.5, -0.5)-- (-0.1,-0.1);
    \draw (.5,-0.5)  -- (0,0)-- (-.5,.5) node[left]{\tiny{c}} -- (-.1,0.9);
    \draw (.1,1.1)--(.5,1.5) node{\tiny{\, b}} --(-.5,2.5)node[left]{\tiny{j}} --(-0.1,2.9);
    \draw (0.1,0.1) -- (.5,.5) node{\tiny{\, d}} --(-.5,1.5) node[left]{\tiny{a}} -- (-.1,1.9);
   \draw (0.1,2.1)--(0.5,2.5) node{\tiny{\, 0}} --(-.5,3.5) node[left]{\tiny{0}};
   \draw (.1, 3.1)--(0.5,3.5) node[left]{\tiny{j}};
   \draw (0.5,-0.5) .. controls (1.5,-.9) and (1.5,3.9) .. (0.5,3.5);
   \end{tikzpicture}  
.$$

For the next crossing, some number of the $j$ walkers can jump down to the lower strand. Let us say $l$ of them stay on the upper strand and $j-l$ jump down.

$$\begin{tikzpicture}[scale=0.8, baseline=0.5cm, rounded corners=2mm]
    \draw (-.5, -0.5)-- (-0.1,-0.1);
    \draw (.5,-0.5)  -- (0,0)-- (-.5,.5) node[left]{\tiny{c}} -- (-.1,0.9);
    \draw (.1,1.1)--(.5,1.5) node{\tiny{\, l}} --(-.5,2.5)node[left]{\tiny{j}} --(-0.1,2.9);
    \draw (0.1,0.1) -- (.5,.5) node{\tiny{\, d}} --(-.5,1.5) node[left]{\tiny{j-l}} -- (-.1,1.9);
   \draw (0.1,2.1)--(0.5,2.5) node{\tiny{\, 0}} --(-.5,3.5) node[left]{\tiny{0}};
   \draw (.1, 3.1)--(0.5,3.5) node[left]{\tiny{j}};
   \draw (0.5,-0.5) .. controls (1.5,-.9) and (1.5,3.9) .. (0.5,3.5);
   \end{tikzpicture}  
.$$

For the next crossing, the same situation arises. Let us say $m$ of them stay on the upper strand, and thus $j-m$ end up on the lower strand.

\begin{figure}[h]
$$\begin{tikzpicture}[scale=0.8, baseline=0.5cm, rounded corners=2mm]
    \draw (-.5, -0.5)-- (-0.1,-0.1);
    \draw (.5,-0.5)  -- (0,0)-- (-.5,.5) node[left]{\tiny{j-m}} -- (-.1,0.9);
    \draw (.1,1.1)--(.5,1.5) node{\tiny{\, l}} --(-.5,2.5)node[left]{\tiny{j}} --(-0.1,2.9);
    \draw (0.1,0.1) -- (.5,.5) node{\tiny{\, m}} --(-.5,1.5) node[left]{\tiny{j-l}} -- (-.1,1.9);
   \draw (0.1,2.1)--(0.5,2.5) node{\tiny{\, 0}} --(-.5,3.5) node[left]{\tiny{0}};
   \draw (.1, 3.1)--(0.5,3.5) node[left]{\tiny{j}};
   \draw (0.5,-0.5) .. controls (1.5,-.9) and (1.5,3.9) .. (0.5,3.5);
   \end{tikzpicture}  
.$$
\end{figure}

However, the lowest labels must be $0$ and $j$, therefore, for the final crossing, no walkers can end on the lower strand, which means there must have been no walkers on the lower strand to begin with. Thus $m=0$.

$$\begin{tikzpicture}[scale=0.8, baseline=0.5cm, rounded corners=2mm]
    \draw (-.5, -0.5)-- (-0.1,-0.1);
    \draw (.5,-0.5)  -- (0,0)-- (-.5,.5) node[left]{\tiny{j}} -- (-.1,0.9);
    \draw (.1,1.1)--(.5,1.5) node{\tiny{\, l}} --(-.5,2.5)node[left]{\tiny{j}} --(-0.1,2.9);
    \draw (0.1,0.1) -- (.5,.5) node{\tiny{\, 0}} --(-.5,1.5) node[left]{\tiny{j-l}} -- (-.1,1.9);
   \draw (0.1,2.1)--(0.5,2.5) node{\tiny{\, 0}} --(-.5,3.5) node[left]{\tiny{0}};
   \draw (.1, 3.1)--(0.5,3.5) node[left]{\tiny{j}};
   \draw (0.5,-0.5) .. controls (1.5,-.9) and (1.5,3.9) .. (0.5,3.5);
   \end{tikzpicture}  
.$$

Now the colored Jones polynomial for the $(2,-4)$ torus link is:
\begin{eqnarray*}
J_{N,T(2,-4)}(q) &=& q^{1-N^2}\sum_{0\leq l \leq j \leq N-1} (R^{-1})^{0~j}_{j~0} (R^{-1})^{j~0}_{j-l~l} (R^{-1})^{j-l~l}_{j~0} (R^{-1})^{j~0}_{0~j} \mu_j\\
&=& q^{1-N^2}\sum_{0\leq l \leq j \leq N-1} \frac{\{N-1-l\}!\{j\}!\{N-1\}!}{\{N-1-j\}!\{l\}!\{j-l\}!\{N-1-j+l\}!}\\
&& \times q^{-(2j-N+1)/2- 3(j-(N-1)/2)(N-1)/2 -(l-j+1)(j-l)/2+(j-l-(N-1)/2)(l-(N-1)/2)}
\end{eqnarray*}

\end{example}

\section{Head and tail of the colored Jones polynomial}

We are interested in the $N$ leading coefficients of the colored Jones polynomial $J_{N,K}(q)$:

\begin{df}
For a Laurent polynomials $P_1(q)$ and a power series $P_2(q)$ we define
$$P_1(q) \dot{=}_n P_2(q)$$
if  $P_1(q)$ coincide - up to multiplication with $\pm q^s$, $s$ some power - with $P_2(q) \mod q^n$.
For example $-q^{-4} + 2 q^{-3}- 3+11 q \dot{=}_5 1-2 q+3 q^4.$
\end{df}

\begin{df}[Head and Tail of the colored Jones polynomial]
The tail of the colored Jones polynomial of a knot $K$ - if it exists - is a series $T_K(q)=\sum_{j=0}^{\infty} a_j q^j$
with
$$J_{N,K}(q) \dot {=}_N T_K(q), \mbox{ for all } N.$$

Similarly the head of the colored Jones polynomial $J_{N,K}(q)$ is defined to be - if it exist - the tail of the colored Jones polynomial of $J_{N,K}(1/q).$ In particular this means that - providing existence - the head of the colored Jones polynomial of a knot $K$ is the tail of the colored Jones polynomial of its mirror image $K^{*}$.
\end{df}

A theorem of the first author gives the existence of the head and tail in certain cases:

\begin{theorem}[\cite{Armond:Walks, Armond:HeadAndTailConjecture}] \label{Armond:Walks} Suppose a link $K$ is
\begin{enumerate}
\item a knot and the closure of a positive braid. Then the tail of $K$ is $T_{K}(q)=1.$ \label{ArmondTheoremPartI} \label{Armond:WalksPositive}
\item an alternating link. Then both the head and the tail exist. 
\end{enumerate}
\end{theorem}

\begin{remark}
\begin{enumerate}
\item
Part (\ref{ArmondTheoremPartI}) of Theorem \ref{Armond:Walks} for braid positive knots cannot be extended to positive knots in general. The knot $7_5$ (see Figure \ref{Fig75}) is an example of a positive knot where the tail is 
$\neq 1$.
\item Champanerkar and Kofman \cite{ChampanerkarKofman:TailFullTwist} showed that if the closed positive braid contains a full twist in the braid group then Theorem \ref{Armond:Walks} (\ref{Armond:WalksPositive}) can be strengthened. 
\end{enumerate}
\end{remark}

\begin{figure}
\includegraphics[width=4.5cm]{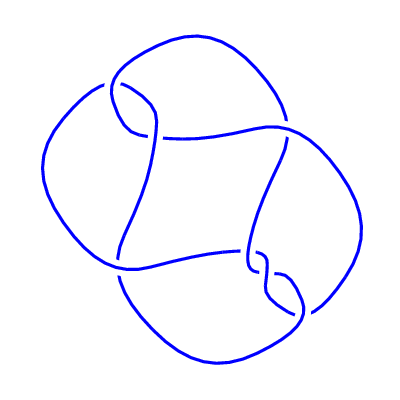}
\caption{The knot $7_5$ is positive \label{Fig75}}
\end{figure}

In fact it follows:

\begin{cor}
Every braid-positive alternating prime knot is a torus knot.
\end{cor}

In \cite{EffieAtAl:GutsAndColoredJones} a generalization of this Corollary is given.

\subsection{Rogers-Ramanujan type identities coming from knots}

The following theorem is essentially a corollary to a theorem of Hugh Morton \cite{Morton:ColoredJonesTorusKnots}:
\begin{theorem}[Left-hand side of the Andrews-Gordon identities (\ref{Andrews-Gordon})]
For a (negative) $(2,2k+1)$-torus knot $K$ the tail is
$$T_K(q)=f(-q^{2k},-q).$$
\end{theorem}

\begin{proof}
Let $p:=2 k+1$. By \cite{Morton:ColoredJonesTorusKnots} we have
\begin{eqnarray*}
(q^N-1) J_{N,K}(q) &\dot{=}& \sum_{r=-(N-1)/2}^{(N-1)/2} q^{p (2 r^2+r)} \left ( q^{2 r+1}-q^{-2r} \right )\\
&\dot{=}& \sum_{r=-(N-1)/2}^{(N-1)/2} q^{p (2 r^2+r)} q^{2 r+1} - q^{p (2 r^2-r)} q^{2 r}\\
&\dot{=}& \sum_{R=-N+1}^N (-1)^R q^{p (R^2-R)/2} q^{R}\\
&\dot{=}& \sum_{R=-N+1}^N (-1)^R q^{k (R^2-R)} q^{(R^2+R)/2}
\end{eqnarray*}

Since $k(R^2-R)+(R^2+R)/2$ is increasing in $|R|$ the result follows from the definition of $f(a,b)$.
\end{proof}

The methods developed by the first author in \cite{Armond:Walks} allow to obtain the other side of the Andrews-Gordon identities:

\begin{theorem}[Right-hand side of the Andrews-Gordon identities (\ref{Andrews-Gordon})]
For a (negative) $(2, 2k+1)$-torus knot $K$ the tail is
$$T_K(q)= (q;q)_{\infty} \sum_{n_1,\dots,n_{k-1}\geq 0} \frac {q^{N_1^2+\dots+N_{k-1}^2+N_1+\dots+N_{k-1}}} {(q;q)_{n_1} \cdots (q;q)_{n_{k-1}}}
$$
with $N_j$ defined as
$$N_j=n_1+\dots+n_{j}.$$
\end{theorem} 

\begin{proof} Let $p=2k+1$ and let $\beta=(\omega_{p-1}\omega_{p-2}\ldots\omega_2\omega_1)^2\in B_p$. Thus $\hat{\beta}= \bar{K}$, the (positive) $(2,2k+1)$-torus knot. Generalizing example \ref{25torus} there are $k$ simple walks along $\beta$.

$$W_j=b_{2(p-1)}\left(\prod_{l=0}^{2(j-1)} c_{p-1-l}\right)a_{p-2j}\left(\prod_{l=1}^{j-1}b_{2(p-1-l)}b_{p-2l}\right)$$
Here since $\beta$ is a positive braid, all of the letters have a $+$ subscript, which we leave off to simplify the notation. One can check that these satisfy the relations $W_iW_j=qW_jW_i$ when $i\leq j$. Therefore we can calculate the colored Jones polynomial of $\bar{K}$, with the notation:

$$\left( \begin{matrix}&n\\n_1  & \ldots & n_{k}\end{matrix} \right)_q:= \frac {(q;q)_{n}} {(q;q)_{n_{1}} \cdots (q;q)_{n_{k}}}$$
where $n=n_1 + \ldots + n_k$.

\begin{eqnarray*}
q^{\frac{-(N-1)(p-1)}{2}}J_{N,\bar{K}}(q)&=&\sum_{n=0}^{N-1} \E_N((W_1+\ldots+W_k)^n)\\
&=&\sum_{n=0}^{N-1}\sum_{\stackrel {0\leq n_1,\ldots,n_k\leq N-1} {n_1+\ldots + n_k=n}}\left( \begin{matrix}&n\\n_1  & \ldots & n_{k}\end{matrix} \right)_q \E_N(W_k^{n_k}\ldots W_1^{n_1})\\
&=&\sum_{n=0}^{N-1}\sum_{\stackrel {0\leq n_1,\ldots,n_k\leq N-1} {n_1+\ldots + n_k=n}}\left( \begin{matrix}&n\\n_1  & \ldots & n_{k}\end{matrix} \right)_q\\ &&\times q^\xi\prod_{j=1}^{n_k}(1-q^{N-j})\prod_{j=1}^{n_{k-1}}(1-q^{N-n_k-j})\ldots\prod_{j=1}^{n_1}(1-q^{N-n_k-n_{k-1}-\ldots-n_2-j})\\
&=&\sum_{n=0}^{N-1}\sum_{\stackrel { 0\leq n_1,\ldots,n_k\leq N-1 } { n_1+\ldots + n_k=n}}\left( \begin{matrix}&n\\n_1  & \ldots & n_{k}\end{matrix} \right)_q q^\xi\prod_{j=0}^{n-1}(1-q^{N-j-1})\\
\end{eqnarray*}

Where 
\begin{eqnarray*}
\xi &=& \sum_{j=1}^{k}((N+1)j-1)n_j + \sum_{j=2}^k (N-1-n_{j-1}-N'_j-1)N'_j\\
N'_j &=& \sum_{i=j}^{k} n_j
\end{eqnarray*}

Since $J_{N,K}(q) = J_{N,\bar{K}}(q^{-1})$, we get

\begin{eqnarray*}
q^{\frac{(N-1)(p-1)}{2}}J_{N,K}(q) &=&\sum_{n=0}^{N-1}\sum_{\stackrel{0\leq n_1,\ldots,n_k\leq N-1} { n_1+\ldots + n_k=n}}\left( \begin{matrix}&n\\n_1  & \ldots & n_{k}\end{matrix} \right)_{q^{-1}} q^{-\xi}\prod_{j=0}^{n-1}(1-q^{-(N-j-1)})\\
&=& \sum_{n=0}^{N-1}\sum_{\stackrel{0\leq n_1,\ldots,n_k\leq N-1 } { n_1+\ldots + n_k=n}}\frac{\prod_{j=1}^n(1-q^{-j})}{\prod_{j=1}^{n_1}(1-q^{-j})\ldots\prod_{j=1}^{n_k}(1-q^{-j})} q^{-\xi}\prod_{j=0}^{n-1}(1-q^{-(N-j-1)})\\
&=& \sum_{n=0}^{N-1}\sum_{\stackrel{0\leq n_1,\ldots,n_k\leq N-1 } { n_1+\ldots + n_k=n}}\frac{\prod_{j=1}^n(1-q^{j})}{\prod_{j=1}^{n_1}(1-q^{j})\ldots\prod_{j=1}^{n_k}(1-q^{j})} (-1)^{\xi''}q^{\xi'}\prod_{j=0}^{n-1}(1-q^{(N-j-1)})\\
\end{eqnarray*}

Here 
$$\xi' = \sum_{j=1}^k \frac{n_j(n_j+1)}{2} - \frac{n(n+1)}{2} - \frac{N(N-1)}{2} + \frac{(N-n)(N-n-1)}{2} - \xi$$ and
$$\xi'' = \sum_{i=1}^{k} n_{i}=n.$$

The lowest $N$ terms all come from terms where $n= N-1$. In particular $(-1)^{\xi''}$ only depends on $N$. So we get

$$J_{N,K}(q)\dot{=}_N\sum_{\stackrel{0\leq n_1,\ldots,n_k\leq N-1 }{ n_1+\ldots +n_k=N-1}}\frac{\prod_{j=1}^{N-1}(1-q^{j})}{\prod_{j=1}^{n_1}(1-q^{j})\ldots\prod_{j=1}^{n_k}(1-q^{j})} q^{\xi'}\prod_{j=1}^{N-1}(1-q^{j})$$

where now $\xi' = \sum_{j=1}^k \frac{n_jn_j}{2} - \sum_{j=1}^k(N+1)jn_j - \sum_{j=2}^k(N-2-N_{j-1})N_j$ 

Since $N_j = \sum_{i=1}^j n_i$ we have $N'_{j}=N-1-N_{j-1}$ and
\begin{eqnarray*}
\xi' &=& \sum_{j=1}^k \frac{n_jn_j}{2} - \sum_{j=1}^k(N+1)(N'_{j}) - \sum_{j=2}^{k}(N-2-(N-1-N_{j-2}))(N-1-N_{j-1})\\
&=&\sum_{j=1}^k \frac{n_jn_j}{2} - \sum_{j=1}^k(N+1)(N-1-N_{j-1}) - \sum_{j=2}^{k} (N_{j-2}-1)(N-1-N_{j-1})
\end{eqnarray*}

We are only interested in $J_{N,K}(q)$ up to multiplication with powers of $q$. Thus we can simplify terms in $\xi'$ that only depend on $N$.
We denote those simplifications by $\sim$.

\begin{eqnarray*}
\xi'&\sim&\sum_{j=1}^k \frac{n_jn_j}{2} + \sum_{j=2}^k(N+1)N_{j-1} - \sum_{j=2}^{k} N_{j-1} - \sum_{j=2}^{k} N_{j-2}(N-1-N_{j-1})\\
&\sim&\sum_{j=1}^k \frac{n_jn_j}{2} + \sum_{j=2}^{k}N \, N_{j-1} - \sum_{j=3}^{k} (N-1)N_{j-2} + \sum_{j=3}^{k} N_{j-1}\, N_{j-2}\\
&\sim&\sum_{j=1}^k \frac{n_jn_j}{2} + N \, N_{k-1} + \sum_{j=2}^{k-1}[N\, N_{j-1}-(N-1)N_{j-1}] + \sum_{j=2}^{k-1} N_{j}\, N_{j-1}\\
&\sim&\sum_{j=1}^k \frac{n_jn_j}{2} + (n_k+N_{k-1}+1)N_{k-1} + \sum_{j=2}^{k-1}N_{j-1} + \sum_{j=2}^{k-1} N_{j} \, N_{j-1}\\
\end{eqnarray*}

Hence
\begin{eqnarray*}
\xi'&\sim&\sum_{j=1}^k \frac{n_jn_j}{2} + n_k \, N_{k-1} + N_{k-1}(N_{k-1}+1) + \sum_{j=2}^{k-1}n_j \, N_{j-1}+ \sum_{j=1}^{k-2}N_{j}(N_{j}+1)\\
&\sim&\sum_{j=1}^k \frac{n_jn_j}{2} + \sum_{j=2}^{k}n_j \, N_{j-1}+ \sum_{j=1}^{k-1}N_{j}(N_{j}+1)\\
&\sim&\sum_{j=1}^k \frac{n_jn_j}{2} + \sum_{j=2}^{k}[n_j \, \sum_{i=1}^{j-1}n_i]+ \sum_{j=1}^{k-1}N_{j}(N_{j}+1)\\
&\sim&\sum_{i=1}^k\sum_{j=1}^k \frac{n_jn_i}{2} + \sum_{j=1}^{k-1}N_{j}(N_{j}+1)\\
&\sim&\frac{(N-1)(N-1)}{2} + \sum_{j=1}^{k-1}N_{j}(N_{j}+1)\\
&\sim&\sum_{j=1}^{k-1}N_{j}(N_{j}+1)\\
\end{eqnarray*}

Thus
$$J_{N,K}(q) \dot{=}_N \sum_{\stackrel{0\leq n_1,\ldots,n_k\leq N-1 } {n_1+\ldots + n_k=N-1}} \frac{\prod_{j=1}^{N-1}(1-q^{j})q^{\sum_{j=1}^{k-1}N_{j}(N_{j}+1)}}{\prod_{j=1}^{n_1}(1-q^{j})\ldots\prod_{j=1}^{n_{k-1}}(1-q^{j})} \frac{\prod_{j=1}^{N-1}(1-q^{j})}{\prod_{j=1}^{n_k}(1-q^{j})}$$

But $\frac{\prod_{j=1}^{N-1}(1-q^{j})}{\prod_{j=1}^{n_k}(1-q^{j})} = \prod_{j=n_k+1}^{N-1}(1-q^{j}) = 1 - q^{n_k + 1} + \text{higher order terms}$. Thus because $\sum_{j=1}^{k-1}N_{j}(N_{j}+1) + n_k + 1 \geq N_k + 1 = N$, we get
$$J_{N,K}(q) \dot{=}_N \prod_{j=1}^{N-1}(1-q^{j})\sum_{\stackrel{0\leq n_1,\ldots,n_{k-1}\leq N-1 } { n_1+\ldots +n_{k-1} = N-1}} \frac{q^{\sum_{j=1}^{k-1}N_{j}(N_{j}+1)}}{\prod_{j=1}^{n_1}(1-q^{j})\ldots\prod_{j=1}^{n_{k-1}}(1-q^{j})}$$

 \end{proof}

\subsection{Identities coming from links}

The following theorem is essentially a corollary to a theorem of Kazuhiro Hikami \cite{Hikami:ColoredJonesTorusLinks} using the methods developed in \cite{Morton:ColoredJonesTorusKnots}.
For links we understand all components to be colored with the same color.

\begin{theorem}
For a (negative) $(2, 2k)$-torus link $L$ the tail is
$$T_{L}(q)=\Psi(q^{2k-1},q).$$
\end{theorem}
\begin{proof}
By \cite{Hikami:ColoredJonesTorusLinks}:
\begin{eqnarray*}
q^{-(N+1)/2}(q^{N+1}-1) J_{N+1,L}(q)&=&q^{-\frac{ k (N^{2}-1)+1}{2}}  \left ( \sum_{r=0}^{N-1} q^{k r^{2}+(k+1)r +1} - \sum_{r=0}^{N-1} q^{k r^{2}+(k-1)r}   \right )
\end{eqnarray*}

Together with
\begin{eqnarray*}
\Psi(q^{2k-1},q)&=&\sum_{r=0}^{\infty} q^{(2k-1) r (r+1)/2} q^{r (r-1)/2} - \sum_{r=1}^{\infty} q^{(2k-1) r (r-1)/2} q^{r (r+1)/2}\\
&=& \sum_{r=0}^{\infty} q^{k r^{2}+(k-1) r}  - \sum_{r=1}^{\infty} q^{k(r^{2}-r)+r}\\
&=& \sum_{r=0}^{\infty} q^{k r^{2}+(k-1) r}  - \sum_{r=0}^{\infty} q^{k r^{2}+(k+1) r +1}
\end{eqnarray*}
the result follows.
\end{proof}

Thus, in particular the tail of the (negative) $(2,4)$-torus link is given by $\Psi(q^3,q).$
On the other hand a closer look at Example \ref{toruslink} gives us

\begin{theorem}
The tail of the colored Jones polynomial for the (negative) $(2,4)$-torus link is given by
$$T_{L}(q)=\Psi(q^3,q)=(q;q)_{\infty}^2 \sum_{k=0}^{\infty} \frac{q^k}{(q;q)^2_k}$$
\end{theorem}

\begin{proof}
Using the fact that $\{m\}! = q^{-m(m+1)/4}(q,q)_m$ we can rewrite the formula from Example \ref{toruslink} to get the following:

\begin{eqnarray*}
J_{T(2,-4),N}(q) &=& q^{1-N^2}\sum_{0\leq l \leq j \leq N-1} \frac{\{N-1-l\}!\{j\}!\{N-1\}!}{\{N-1-j\}!\{l\}!\{j-l\}!\{N-1-j+l\}!}\\
&& \times q^{-(2j-N+1)/2- 3(j-(N-1)/2)(N-1)/2 -(l-j+1)(j-l)/2+(j-l-(N-1)/2)(l-(N-1)/2)}\\
&=&q^{1-N^2}\sum_{0\leq l \leq j \leq N-1} \frac{(q;q)_{N-1-l}(q;q)_{j}(q;q)_{N-1}}{(q;q)_{N-1-j}(q;q)_{l}(q;q)_{j-l}(q;q)_{N-1-j+l}}\\
&& \times q^{(1 - 3 N)/2 + j + j^2 - 3 j N + l (N - j) + N^2}\\
\end{eqnarray*}

Note that among all the terms with a fixed value for $j$, the lowest degree comes from the terms where $l=0$. So, restricting to $l=0$, the minimum degree decreases as $j$ increases. Also the difference between the minimum degree when $j=N-2$ and when $j=N-1$ is $N+2$, thus all terms which contribute to the tail have $j=N-1$.

\begin{eqnarray*}
J_{T(2,-4),N}(q) &\dot{=}_N& \sum_{l=0}^{N-1} \frac{(q;q)_{N-1-l}(q;q)_{N-1}(q;q)_{N-1}}{(q;q)_{l}(q;q)_{N-1-l}(q;q)_{l}} q^{l}\\
&=&(q;q)_{N-1}^2\sum_{l=0}^{N-1} \frac{q^{l}}{(q;q)_{l}^2} \\
&\dot{=}_N&(q;q)_{\infty}^2\sum_{l=0}^{\infty} \frac{q^{l}}{(q;q)_{l}^2} \\
\end{eqnarray*}
\end{proof}

\section{The head and tails for alternating links only depend on the reduced checkerboard graphs}
\label {DependenceOnReducedCheckerboardGraph}

Without loss of generality we assume that all alternating knot diagrams are reduced, i.e. they do not contain nugatory crossings. 
Given an alternating knot $K$ with alternating diagram $D$ a checkerboard (black/white) shading of the faces of $D$ defines two (plane) checkerboard graphs in a natural way: The vertices of the first (second) graph are given by the white (black) faces and two vertices are connected by an edge if they meet at a crossing of the diagram. The two graphs can be distinguished from each other in the following way: The $A$-checkerboard graph is the graph where the edges correspond to arcs that overcross from the right to the left and correspondingly the $B$-checkerboard graph is the graph where the edges correspond to arcs that overcross from the left to the right. In particular the $A$-checkerboard graph of the diagram is the $B$-checkerboard graph of the diagram of the mirror image and vice versa.
The two graphs are dual to each other. We obtain the reduced checkerboard graphs by replacing parallel edges by single edges. Note, that it is easy to see that by applying flypes to the knot diagram we can assume that parallel edges in the graph are also parallel in the embedding. For the knot itself this reduction means that suitable $k$-half twists
are replaced by a single half-twist. Figure \ref{Fig920} gives an example. Note, that in general the two reduced checkerboard graphs are not dual to each other anymore, and one cannot be constructed from the other.

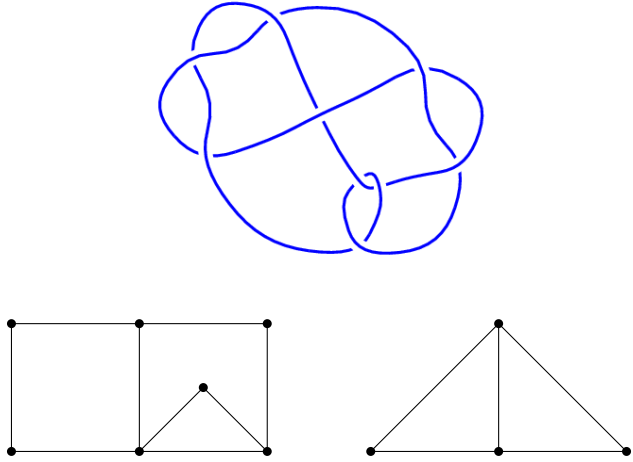
\begin{figure}
\includegraphics[width=5cm]{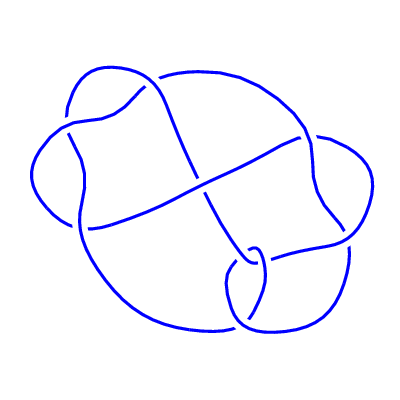}

\begin{tikzpicture}[scale=1.7] 
\path (0,1) coordinate (X1); 
\path (-1,1) coordinate (X2);
\path (-1,0) coordinate (X3);
\path (0,0) coordinate (X4);
\path (1,0) coordinate (X5);
\path (1,1) coordinate (X6);
\path (0.5,0.5) coordinate (X7);
\foreach \j in {1, ..., 7} \fill (X\j) circle (1pt); 
\draw (X1)--(X2)--(X3)--(X4)--(X5)--(X6)--(X1)--(X4)--(X7)--(X5);
\end{tikzpicture} \hspace{1cm}
\begin{tikzpicture}[scale=1.7]
\path (0,0) coordinate (X1);
\path (1,0) coordinate (X2);
\path (0,1) coordinate (X3);
\path (-1,0) coordinate (X4);
\foreach \j in {1, ..., 4} \fill (X\j) circle (1pt); 
\draw (X1)--(X2)--(X3)--(X4)--(X1)--(X3);
\end{tikzpicture}
\caption{The knot $9_{20}$ and its two reduced checkerboard graphs \label{Fig920}. The reduced $B$-checkerboard graph is on the left and the reduced $A$-checkerboard graph on the right}
\end{figure}

\begin{remark} In \cite{DasbachLin:HeadAndTail} (and compare with \cite{DL:VolumeIsh}) it was shown that for an alternating link $K$ with diagram $D$ and reduced $A$-checkerboard graph $G$ the colored Jones polynomial satisfies:
$$J_{N,K}\dot{=}_3 \, 1-a q + b q^2, \mbox { for } N \geq 3,$$
where 
$a=\beta_1(G)$, the first Betti number of $G$ and $$b={a \choose 2}-t(G)$$with $t(G)$ the number of triangles in $G$.
For the knot $9_{20}$ in Figure \ref{Fig920} $a=2$, $t(G)=2$ thus $b={2 \choose 2} - 2= 1-2=-1$ and the colored Jones polynomial for $N \geq 3$ starts with:
$$J_{N,K}\dot{=}_3 \, 1-2 q - q^2.$$

As an application \cite{DasbachLin:HeadAndTail} one sees that the volume of the  hyperbolic link complement of an alternating link is bounded from above and below linearly in the absolute values of the second and the penultimate coefficient of the colored Jones polynomial. This was generalized to other classes of knots and links by Futer, Kalfagianni and Purcell \cite{FKP:VolumeJones}.
\end{remark}

The property that the first three coefficients of the colored Jones polynomial only depend on the reduced $A$-checkerboard graph of the knot diagram holds for the whole tail of the colored Jones polynomial:

\begin{theorem}
Let $K_1$ and $K_2$ be two alternating links with alternating diagrams $D_1$ and $D_2$ such that the reduced $A$-checkerboard (respectively $B$ checkerboard) graphs of $D_1$ and $D_2$ coincide. Then
the tails (respectively heads) of the colored Jones polynomial of $K_1$ and $K_2$ are identical.
\end{theorem}

\begin{proof}
First recall the computation of the colored Jones polynomial via skein theory as in Section \ref{SectionSkeinTheory}. By convention $q=A^{-4}$. Thus the meaning of head and tail interchange if we switch from the variable $A$ to $q$. A negative twist region in the knot diagram with $m$ negative half-twists correspond to $m$ parallel edges in the $B$-checkerboard graph:

$$\begin{tikzpicture}[scale=0.6, baseline=0.7cm, rounded corners=2mm]
    \draw (-.5,2) -- (-0.1, 2.4);
    \draw (0.1,2.6) -- (0.5,3);
    \draw (0.5, 2) -- (-0.5,3);
    \draw[dashed] (0, 1.5) -- (0,1.75) node[left]{\small{m}} --(0, 2);
    \draw (-.5, -0.5)  -- (-0.1,-0.1);
    \draw (.5,-0.5) -- (0,0)-- (-.5,.5) -- (-.1,0.9);
    \draw (.1,1.1)--(.5,1.5);
      \draw (0.1,0.1) -- (.5,.5)--(-.5,1.5) ;
\end{tikzpicture} \leadsto \, \, 
\begin{tikzpicture}[scale=0.6, baseline=0.7cm]
\draw (-3, 1.25) -- (0,1.25) node[above]{\small{m}} -- (3, 1.25);
\draw[dashed] (0, 2) -- (0, 2.8);
\draw[dashed] (0, 1.24) -- (0, -0.4);
\draw (-3, 1.25) .. controls (-2.5, 3.5) and (2.5,3.5) .. (3,1.25);
\draw (-3, 1.25) .. controls (-2.5, -1) and (2.5, -1) .. (3,1.25);
\end{tikzpicture}
$$

Thus to prove the theorem it is sufficient to show that the tail of  colored Jones polynomial in the variable $A$ is invariant under negative twists.
This will be done in the remainder of this section.
\end{proof}

We appeal to the notations of Section \ref{SectionSkeinTheory}.
Given an alternating diagram $D$ of a link $L$ and consider a negative twist region. Apply the identities of Section \ref{SectionSkeinTheory} to get the equation:

$$\begin{tikzpicture}[scale=0.6, baseline=0.8cm, rounded corners=2mm]
    \draw (-.5,2) -- (-0.1, 2.4);
    \draw (0.1,2.6) -- (0.5,3);
    \draw (0.5, 2) -- (-0.5,3);
    \draw[dashed] (0, 1.5) -- (0,1.75) node[left]{\small{m}} --(0, 2);
    \draw (-.5, -0.5) node[left]{\tiny{n}} -- (-0.1,-0.1);
    \draw (.5,-0.5) node[right]{\tiny{n}}-- (0,0)-- (-.5,.5) -- (-.1,0.9);
    \draw (.1,1.1)--(.5,1.5);
      \draw (0.1,0.1) -- (.5,.5)--(-.5,1.5) ;
\end{tikzpicture} \, = \sum_{j=0}^n (\gamma(n,n,2j))^m \frac{\Delta_{2j}}{\theta(n,n,2j)} 
\begin{tikzpicture}[scale=0.6, baseline=0.8cm]
    \draw (-.5,3) node[left]{\tiny{n}}-- (0, 2.5);
    \draw (0.5,3) node[right]{\tiny{n}}-- (0,2.5);
    \draw (0, 2.5) --(0,1.25) node[right]{\tiny{2j}}-- (0,0);
    \draw (0,0)--(-.5,-.5) node[left]{\tiny{n}};
    \draw (0,0)--(.5,-.5) node[right]{\tiny{n}};
\end{tikzpicture}
.$$

Here $\gamma(a,b,c):=(-1)^{\frac {a+b-c} 2} A^{a+b-c+ \frac{a^2+b^2-c^2} 2}.$

We would like to say that the tail of the left-hand side is equivalent to the tail of the right-hand side with the sum removed and $j$ replaced with $n$. However this statement is difficult to consider. Instead we will apply this operation to every maximal negative twist region to get a trivalent graph $\Gamma$. We will get a colored graph $\Gamma_{n,(j_1,\ldots,j_k)}$ where $k$ is the number of maximal negative twist regions and $0\leq j_i \leq n$ by coloring the edge coming from the $i$-th twist region by $j_i$ and coloring all of the other edges by $n$. Figure \ref{CodysFavorite} gives an example.

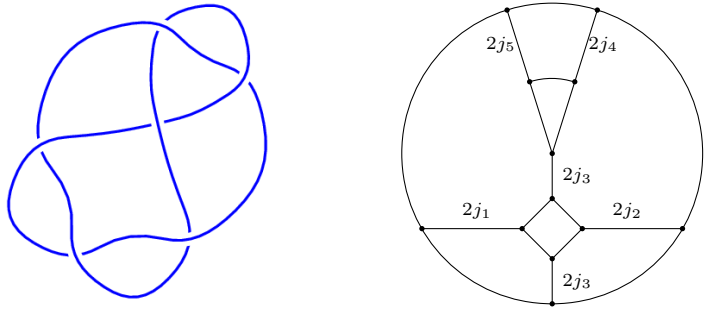
\begin{figure}
  \includegraphics[width=1.8in]{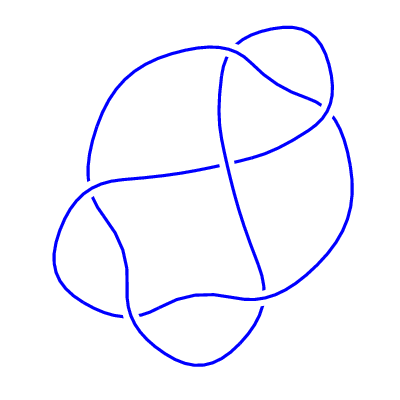} 
  \hspace {1cm}
\begin{tikzpicture}[baseline=-2.3cm, scale=1]
\draw (0,0) circle (2);
\path (0,-2) coordinate (X0);
\path (-1.732, -1) coordinate (X1);
\path (1.732, -1) coordinate (X2);
\path (-0.4, -1) coordinate (X3);
\path (0.4, -1) coordinate (X4);
\path (0, -0.6) coordinate (X6);
\path (0,-1.4) coordinate (X5);
\path (0, 0) coordinate (X7);
\path (-0.3,0.953939) coordinate (X8);
\path (0.3, 0.953939) coordinate (X9);
\path (-0.6, 1.90788) coordinate (X10);
\path (0.6, 1.90788) coordinate (X11);
\foreach \j in {0, ..., 11} \fill (X\j) circle (1pt); 
\draw (X1)--(-1, -1) node[above]{\tiny{$2j_1$}}-- (X3);
\draw (X4)--(1,-1) node[above]{\tiny{$2j_2$}}--(X2);
\draw (X0)--(0,-1.7) node[right]{\tiny{$2j_3$}}--(X5);
\draw (X5)--(X4);
\draw (X5)--(X3);
\draw (X3)--(X6);
\draw (X6)--(X4);
\draw (X6)--(0,-0.3) node[right]{\tiny{$2j_3$}} -- (X7);
\draw (X7)--(X10);
\draw (X7)--(X11);
\draw (X8) arc (107:72:1);
\draw (0.68, 1.45) node{\tiny{$2j_4$}};
\draw (-0.68, 1.45) node{\tiny{$2 j_5$}};
\end{tikzpicture}
\caption{The knot $6_2$ and the corresponding $\Gamma_{n,(j_1, \dots, j_5 )}$. All missing labels are $n$ \label{CodysFavorite}}
\end{figure}

\begin{theorem}\label{skein}
$$\tilde J_{n+1,L} \dot{=}_{4(n+1)} \Gamma_{n,(n,\ldots,n)}$$
\end{theorem}

For the proof first note that by applying the previous equation $k$ times we get

$$\tilde J_{n+1,L} \dot{=} \sum_{j_1,\ldots,j_k=0}^n \prod_{i=1}^k (\gamma(n,n,2j_i))^m \prod_{i=1}^k \frac{\Delta_{2j_i}}{\theta(n,n,2j_i)} \Gamma_{n,(j_1,\ldots,j_k)}$$

For a rational function $R$, let $d(R)$ be the minimum degree of $R$ considered as a power series. the theorem will now follow from the following three lemmas.

\begin{lem}
\label{gamma}
\begin{eqnarray*}
d(\gamma(n,n,2n)) &\leq& d(\gamma(n,n,2(n-1))) - 4n\\
d(\gamma(n,n,2j)) &\leq& d(\gamma(n,n,2(j-1)))
\end{eqnarray*}
\end{lem}

\begin{lem}
\label{fuse}
$$d\left(\frac{\Delta_{2j}}{\theta(n,n,2j)}\right) = d\left(\frac{\Delta_{2(j-1)}}{\theta(n,n,2(j-1))}\right) - 2$$
\end{lem}

\begin{lem}
\label{graph}
\begin{eqnarray*}
d(\Gamma_{n,(j_1,\ldots,j_{(i-1)},j_i,j_{i+1},\ldots,j_k)}) &\leq& d(\Gamma_{n,(j_1,\ldots,j_{(i-1)},j_i-1,j_{i+1},\ldots,j_k)}) \pm 2\\
d(\Gamma_{n,(n,\ldots,n,\ldots,n)}) &\leq& d(\Gamma_{n,(n,\ldots,n-1,\ldots,n)}) - 2
\end{eqnarray*}
\end{lem}

\begin{proof}[Proof of Lemma \ref{gamma}]
\begin{eqnarray*}
\gamma(n,n,2j) &=& \pm A^{n+n-2j+\frac{n^2+n^2-(2j)^2}{2}}\\
&=& \pm A^{2n-2j+n^2-2j^2}
\end{eqnarray*}
Clearly $d(\gamma(n,n,2j))$ increases as $j$ decreases. Furthermore:
\begin{eqnarray*}
d(\gamma(n,n,2n)) &=& -n^2\\
d(\gamma(n,n,2(n-1))) &=& 2n-2(n-1)+n^2-2(n-1)^2\\
&=&-n^2 + 4n\\
\end{eqnarray*}

\end{proof}

\begin{proof}[Proof of Lemma \ref{fuse}]

To calculate $\theta(n,n,2j)$ note that in the previous formula for $\theta$ we get $x= j$, $y=j$, and $z=n-j$. Using this and the fact that $d(\Delta_n) = -2n$, we get:

\begin{eqnarray*}
d\left(\frac{\Delta_{2j}}{\theta(n,n,2j)}\right) &=& d\left(\frac{\Delta_{2j}\Delta_{n-1}!\Delta_{n-1}!\Delta_{2j-1}!}{\Delta_{n+j}!\Delta_{j-1}!\Delta_{j-1}!\Delta_{n-j-1}!}\right)\\
&=& d\left(\frac{\Delta_{2j}\Delta_{2j-1}\Delta_{n-j}}{\Delta_{n+j}\Delta_{j-1}\Delta_{j-1}}\right) + d\left(\frac{\Delta_{2(j-1)}!\Delta_{n-1}!\Delta_{n-1}!}{\Delta_{n+j-1}!\Delta_{j-2}!\Delta_{j-2}!\Delta_{n-j}!}\right)\\
&=& -4j -2(2j-1) - 2(n-j)+4(j-1) +2(n+j) + d\left(\frac{\Delta_{2(j-1)}}{\theta(n,n,2(j-1))}\right)\\
&=& -2 + d\left(\frac{\Delta_{2(j-1)}}{\theta(n,n,2(j-1))}\right)
\end{eqnarray*}

\end{proof}

Before we can prove Lemma \ref{graph}, we need to consider a property of the Jones-Wenzl idempotent. This idempotent is a linear combination of crossingless matchings,  where the coefficients are rational functions in $A$. It is well-known that the crossingless matchings form a monoid, which we will call the Temperley-Lieb monoid (this monoid is very similar to the Temperley-Lieb Algebra). The Temperley-Lieb Monoid is generated by elements $h_i$ called hooks.

\begin{proposition}
\label{idem}
The coefficient of the crossingless matching $M$ in the expansion of the Jones-Wenzl idempotent has minimum degree at least twice the minimum word length of $M$ in terms of the $h_i$'s.
\end{proposition}

\begin{proof}
This follows easily from the recursive definition of the idempotent by an inductive argument. The only issue is that in terms of the form 
$\begin{tikzpicture}[baseline=0.8cm]
\draw (0,0) node[right]{\small{n}}--(0,1) node[left]{\scriptsize{n-1}}
--(0,2) node[right]{\small{n}};
\draw[ultra thick, fill=white] (-0.2 ,0.5) rectangle (0.4, 0.65); 
\draw (0.2, 0.65) arc (180:0:0.25);
\draw (0.7,0) node[right]{1}--(0.7, 0.65);
\draw[ultra thick, fill=white] (-0.2 ,1.5) rectangle (0.4, 1.35); 
\draw (0.2, 1.35) arc (-180:0:0.25);
\draw (0.7,2) node [right]{\small{1}} --(0.7, 1.35);
\end{tikzpicture}$ may have a circle which needs to be removed. In this situation, the minimum degree of the coefficient is reduced by two, but the number of generators used is also reduced by one.

\end{proof}

\begin{proof}[Proof of Lemma \ref{graph}]
Consider the graph $\Gamma_{n,(j_1,\ldots,j_k)}$ viewed as an element in the Kauffman Bracket Skein Module of $S^3$. We can expand each of the Jones-Wenzl idempotents that appear as in Proposition \ref{idem}. Consider a single term $T_1$ in the expansion. Unless all of the idempotents have been replaced by the identity in this term, then there will be a hook somewhere in the diagram. By removing a single hook, we get different term $T_2$ in the expansion. The number of circles in $T_1$ differs from the number of circles in $T_2$ by exactly one. Also there are fewer hooks in $T_2$, so by Proposition \ref{idem}, the minimum degree of $T_1$ is at least as large as the minimum degree of $T_2$. If, however, $T_2$ is the term with no hooks, then $T_2$ has one more circle than $T_1$ and thus $d(T_2) \leq d(T_1) -4$. This argument implies that the minimum degree of $\Gamma_{n,(j_1,\ldots,j_k)}$ comes from the term with each idempotent replaced with the identity.

Now we must compare $d(\Gamma_{n,(j_1,\ldots,j_{(i-1)},j_i,j_{i+1},\ldots,j_k)})$ with $d(\Gamma_{n,(j_1,\ldots,j_{(i-1)},j_i-1,j_{i+1},\ldots,j_k)})$. From the previous paragraph, we know that both minimum degrees come from terms with the identities plugged into the idempotents. But comparing these two terms coming from these graphs, we see that the coefficient is $1$ in both cases, but the number of circles differ by $1$. So we get
$$d(\Gamma_{n,(j_1,\ldots,j_{(i-1)},j_i,j_{i+1},\ldots,j_k)}) \leq d(\Gamma_{n,(j_1,\ldots,j_{(i-1)},j_i-1,j_{i+1},\ldots,j_k)}) \pm 2$$

Finally $\Gamma_{n,(n,\ldots,n,\ldots,n)}$ has one more circle than $\Gamma_{n,(n,\ldots,n-1,\ldots,n)}$, so we get 
$$d(\Gamma_{n,(n,\ldots,n,\ldots,n)}) \leq d(\Gamma_{n,(n,\ldots,n-1,\ldots,n)}) - 2$$
\end{proof}

\section{The monoid of heads and tails of prime alternating links} \label{ProductFormula}

The head and the tails of the colored Jones polynomials of prime alternating links form a monoid in the following sense:

\begin{theorem} \label{Thm:monoid} Let $K_1$ and $K_2$ be two prime alternating links. Then there exist a prime alternating link $K_3$ such that the tails of the colored Jones polynomials of the three knots satisfy:
$$T_{K_1}(q) T_{K_2}(q) \dot{=} T_{K_3}(q).$$
In particular any alternating link $K_3$ whose reduced $A$-graph can be formed by gluing the reduced $A$-graphs of $K_1$ and $K_2$ along a single edge (as in Figure \ref{prod}) satisfies this statement.

A corresponding result holds for the heads.
\end{theorem}

\begin{proof}
The proof of this theorem uses Theorem \ref{skein}, so because in the skein picture $q=A^{-4}$ we will consider the mirror images $K_1^*, K_2^*$ and $K_3^*$ of $K_1$, $K_2$, and $K_3$. Thus it is their reduced $B$-graphs which are related as in the statement of the theorem.

In Figures \ref{pairing} and \ref{closure}, the interior of the dotted regions represent $S(D^2,N,N,2 N)$ the Kauffman Bracket Skein Module of the disk with three colored points. This is known to be one dimensional when the three colored points are admissibly colored, generated by a single trivalent vertex \cite{Lickorish:KnotTheoryBook}. Thus any element of $S(D^2,N,N,2N)$ is some rational function times the generator. Let $\bar{\Gamma}$ be the closure of $\Gamma$ by filling in the outside of the dotted circle by a single trivalent vertex as in Figure \ref{closure}. Also define a bilinear pairing $<\Gamma_1, \Gamma_2>$ which identifies the boundaries of $\Gamma_1$ and $\Gamma_2$ as in Figure \ref{pairing}.

By Theorem \ref{skein}, there is a $\Gamma_1$ and $\Gamma_2$, such that
\begin{eqnarray*}
\Delta_n J_{K_1^*,N+1} &\dot{=}_{4(N+1)}& \bar{\Gamma}_1\\
\Delta_n J_{K_2^*,N+1} &\dot{=}_{4(N+1)}& \bar{\Gamma}_2\\
\Delta_n J_{K_3^*,N+1} &\dot{=}_{4(N+1)}& <\Gamma_1,\Gamma_2>\\
\end{eqnarray*}

By the fact that $S(D^2,N,N,2N)$ is one-dimensional, there are rational functions $R_i$ for $i=1,2$ such that $\Gamma_i = R_i *$ the trivalent vertex. Thus 
\begin{eqnarray*}
\bar{\Gamma}_1 &=& R_1 \theta(N,N,2N)\\
\bar{\Gamma}_2 &=& R_2 \theta(N,N,2N)\\
<\Gamma_1,\Gamma_2> &=& R_1 R_2 \theta(N,N,2N)\\
\end{eqnarray*}

Therefore
\begin{eqnarray*}
J_{K_1^*,N+1} &\dot{=}_{4(N+1)}& R_1 R_2 \left(\frac{\theta(N,N,2N)}{\Delta_N}\right)^2\\
&=&R_1 R_2 \left(\frac{\Delta_{2N}}{\Delta_N}\right)^2\\
&=&R_1 R_2 \left(\frac{A^{2(2N+1)}-A^{-2(2N+1)}}{A^{2(N+1)}-A^{-2(N+1)}}\right)^2\\
&\dot{=}_{4(N+1)}& J_{K_3^*,N+1}
\end{eqnarray*}

This last line is true because $\frac{A^{2(2N+1)}-A^{-2(2N+1)}}{A^{2(N+1)}-A^{-2(N+1)}}\dot{=}_{4(N+1)} 1$.

\end{proof}

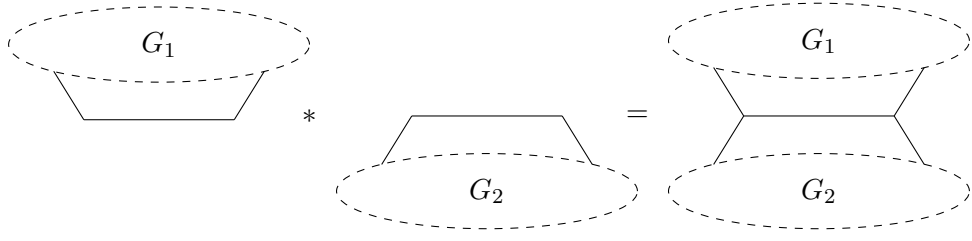
\begin{figure}
\begin{tikzpicture}[scale=1]
\draw[style=dashed]  (0,1) ellipse (2 and 0.5) node{$G_{1}$};
\draw (-1,0) -- (1,0);
\draw (1,0) -- (1.4,0.645);
\draw (-1,0)-- (-1.4,0.645);
\draw [color=white] (0,-1.45);
\draw (2,0) node {*};
\end{tikzpicture}
\begin{tikzpicture}[scale=1]
\draw[style=dashed]  (0,-1) ellipse (2 and 0.5) node{$G_{2}$};
\draw (-1,0) -- (1,0);
\draw (1,0) -- (1.4,-0.645);
\draw (-1,0)-- (-1.4,-0.645);
\draw [color=white] (0,1.45);
\draw (2,0) node {=};
\end{tikzpicture}
\begin{tikzpicture}[scale=1]
\draw[style=dashed]  (0,1) ellipse (2 and 0.5) node{$G_{1}$};
\draw (-1,0) -- (1,0);
\draw (1,0) -- (1.4,0.645);
\draw (1,0)-- (1.4,-0.645);
\draw (-1,0)-- (-1.4, 0.645);
\draw (-1,0)-- (-1.4,-0.645);
\draw[style=dashed]  (0,-1) ellipse (2 and 0.5) node{$G_{2}$};
\end{tikzpicture}
\caption{Product of two checkerboard graphs}
\label{prod}
\end{figure}

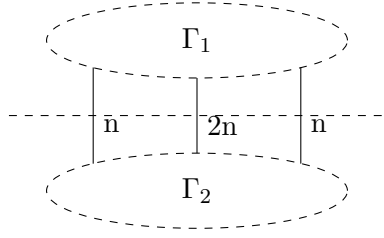
\begin{figure}
\begin{tikzpicture}[scale=1]
\draw[style=dashed]  (0,1) ellipse (2 and 0.5) node{$\Gamma_{1}$};
\draw[style=dashed] (-2.5,0) -- (2.5,0);
\draw (1.38,-0.64) -- (1.38, -0.14) node[right] {n}-- (1.38,0.64);
\draw (0,-0.5)--(0,-0.14) node[right]{2n} --   (0,0.5);
\draw (-1.38,-0.64) -- (-1.38 ,-0.14) node[right]{n}  --(-1.38,0.64);
\draw[style=dashed]  (0,-1) ellipse (2 and 0.5) node{$\Gamma_{2}$};
\end{tikzpicture}
\caption{The bilinear pairing $<\Gamma_1,\Gamma_2>$}
\label{pairing}
\end{figure}

\begin{figure}[ht]
\begin{tikzpicture}[scale=1]
\draw[style=dashed]  (0,1) ellipse (2 and 0.5) node{$\Gamma$};
\draw (0,-0.5)--(0,0.14) node[right]{2n} --   (0,0.5);
\fill (0,-0.5) circle(0.07);
\draw (-1,0.56) arc (-183:3:1);
\draw (1.0,0.14) node[right]{n};
\draw (-0.9,0.14) node[right]{n};
\end{tikzpicture}
\caption{The closure of $\Gamma$}
\label{closure}
\end{figure}
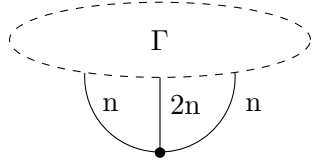

\begin{example}
Figure \ref{Fig920} depicts the knot $9_{20}$ together with its two reduced checkerboard graphs.
In the sense of the proof of Theorem \ref{Thm:monoid} the first reduced checkerboard graph is the product of two squares and a triangle. The second reduced checkerboard graph
is the product of two triangles.

Thus the head and tail functions of the colored Jones polynomial of $9_{20}$ are:
$$T_{9_{20}}(q)\dot{=} \, f(-q^2,-q)^2$$ and
$$H_{9_{20}}(q)\dot{=} \, \Psi(q^3,q)^2 f(-q^2,-q).$$
\end{example}

\section{Multiple Heads}

We single out a sample result for certain non-alternating knots:

\begin{proposition} \label{multihead}
Let $p>m$. A $(m,p)$-torus knot has one head and one tail if $m=2$ and two heads and one tail if $m>2$.
The two heads correspond to even or odd $N$.
\end{proposition}

\begin{proof}
The torus knot is the closure of a positive braid. Thus, by Theorem \ref{Armond:Walks} the tail is identical to $1$.

For the head we have:

\begin{eqnarray*}
(q^N-1) J_{N,K}(q)&\dot =& \sum_{r=-(N-1)/2}^{(N-1)/2} q^{r^2 m p + r p +r m +1}-q^{r^2 m p - r p + r m}\\
&\dot =& \sum_{r=-(N-1)/2}^{(N-1)/2}   q^{\psi_{m,p}(r+ 1/m)} - q^{\psi_{m,p}(r)}
\end{eqnarray*}

with $\psi_{m,p}(r)=r^2 m p - r p + r m.$ Here $\dot{=}$ means up a sign and up to some power of $q^{1/2}$.
\end{proof}

\begin{theorem}
For $K$ a $(4,p)$-torus knot we have \cite{Morton:ColoredJonesTorusKnots}:
$$H_{K}^{\mbox{odd}}(q^2) - q^{p-2} H_{K}^{\mbox{even}}(q^2) = f(-q^2,-q^{p-2}).$$
\end{theorem}
\begin{proof}
By Proposition \ref{multihead} we have:
\begin{eqnarray*}
H_{K}^{\mbox{odd}}(q) - q^{p/2-1} H_{K}^{\mbox{even}}(q) &=& \sum_{r=-\infty}^{\infty} q^{\psi_{4.p}(r)} - q^{\psi_{4,p}(r+1/4)}+q^{\psi_{4,p}(r+1/2)}-q^{\psi_{4,p}(r+3/4)}\\
&=& \sum_{R=-\infty}^{\infty} (-1)^R q^{\psi_{4,p}(R/4)}
\end{eqnarray*}
\end{proof}

\bibliography{HeadAndTail}

\providecommand{\bysame}{\leavevmode\hbox to3em{\hrulefill}\thinspace}
\providecommand{\MR}{\relax\ifhmode\unskip\space\fi MR }
\providecommand{\MRhref}[2]{%
  \href{http://www.ams.org/mathscinet-getitem?mr=#1}{#2}
}
\providecommand{\href}[2]{#2}
\begin{thebibliography}{Arm11b}

\bibitem[AB05]{AndrewsBerndt:RamanujansLostNotebookI}
George~E. Andrews and Bruce~C. Berndt, \emph{{Ramanujan's Lost Notebook: Part
  I., Volume 1}}, Springer, 2005.

\bibitem[Arm11a]{Armond:HeadAndTailConjecture}
Cody Armond, \emph{{The head and tail conjecture for alternating knots}}, 2011.

\bibitem[Arm11b]{Armond:Walks}
\bysame, \emph{{Walks along braids and the colored Jones polynomial}},
  arXiv:1101.3810 (2011), 1--26.

\bibitem[Ber91]{RamanujanNotebooks:PartIII}
Bruce~C. Berndt, \emph{{Ramanujan's notebooks, Part III}}, Springer Verlag, New
  York, 1991.

\bibitem[BN11]{BarNatan:KnotTheory}
Dror Bar-Natan, \emph{{KnotTheory, http://katlas.org}}, 2011.

\bibitem[CK10]{ChampanerkarKofman:TailFullTwist}
Abhijit Champanerkar and Ilya Kofman, \emph{{On the tail of Jones polynomials
  of closed braids with a full twist}}, arXiv \textbf{1004.2694v} (2010),
  1--13.

\bibitem[CL11]{ChaLivingston:KnotInfo}
Jae~Choon Cha and Charles Livingston, \emph{{KnotInfo: Table of Knot
  Invariants, available at: http://www.indiana.edu/\~{}knotinfo}}, 2011.

\bibitem[DL06]{DasbachLin:HeadAndTail}
Oliver~T. Dasbach and Xiao-Song Lin, \emph{{On the head and the tail of the
  colored Jones polynomial}}, Compositio Mathematica \textbf{142} (2006),
  no.~05, 1332--1342.

\bibitem[DL07]{DL:VolumeIsh}
\bysame, \emph{{A volumish theorem for the Jones polynomial of alternating
  knots}}, Pacific J. Math. \textbf{231} (2007), no.~2, 279--291.

\bibitem[FKP08]{FKP:VolumeJones}
David Futer, Efstratia Kalfagianni, and Jessica~S. Purcell, \emph{{Dehn
  filling, volume, and the Jones polynomial}}, J. Differential Geom.
  \textbf{78} (2008), no.~3, 429--464.

\bibitem[FKP11]{EffieAtAl:GutsAndColoredJones}
\bysame, \emph{{Guts of surfaces and the colored Jones polynomial}}, preprint
  (2011).

\bibitem[Hab02]{Habiro:Cyclotomic}
Kazuo Habiro, \emph{{On the quantum sl(2) invariants of knots and integral
  homology spheres}}, Geometry \& Topology Monographs, vol.~4, 2002,
  pp.~55--68.

\bibitem[Hab10]{Habiro:KnotsInPoland}
\bysame, \emph{{On certain limits of the reduced colored Jones polynomials of
  knots}}, Talk at Knots in Poland, Bedlewo, Poland, 2010.

\bibitem[Hik03]{Hikami:VCAsymptoticExpansion}
Kazuhiro Hikami, \emph{{Volume conjecture and asymptotic expansion of
  q-series}}, Experiment. Math. \textbf{12} (2003), no.~3, 319--338.

\bibitem[Hik04]{Hikami:ColoredJonesTorusLinks}
\bysame, \emph{{Quantum Invariant for Torus Link and Modular Forms}},
  Communications in Mathematical Physics \textbf{246} (2004), no.~2, 403--426.

\bibitem[HL07]{VuLe:ColoredJonesDeterminant}
Vu~Huynh and Thang T.~Q. L\^{e}, \emph{{On the colored Jones polynomial and the
  Kashaev invariant}}, J. Math. Sci. (N. Y.) \textbf{146} (2007), no.~1,
  5490--5504.

\bibitem[KM91]{KM:cablingFormulaCoJP}
Robion Kirby and Paul~M. Melvin, \emph{{The 3-manifold invariants of Witten and
  Reshetikhin-Turaev for sl(2, C)}}, Invent. Math. \textbf{105} (1991), no.~1,
  473--545.

\bibitem[Lic97]{Lickorish:KnotTheoryBook}
W.~B.~Raymond Lickorish, \emph{{An introduction to knot theory}}, Springer,
  1997.

\bibitem[LW98]{LinWang:RandomWalkColoredJones}
Xiao-Song Lin and Zhenghan Wang, \emph{{Random walk on knot diagrams, colored
  Jones polynomial and Ihara-Selberg zeta function}}, arXiv (1998), 1--18.

\bibitem[LZ99]{LawrenceZagier:ModularFormsQuantumInvariants}
Ruth Lawrence and Don Zagier, \emph{{Modular Forms and Quantum Invariants of
  3-Manifolds}}, Asian J. Math. \textbf{3} (1999), no.~1, 93----107.

\bibitem[Mor95]{Morton:ColoredJonesTorusKnots}
Hugh~R. Morton, \emph{{The coloured Jones function and Alexander polynomial for
  torus knots}}, Mathematical Proceedings of the Cambridge Philosophical
  Society \textbf{117} (1995), no.~01, 129--135 (English).

\bibitem[MSZ08]{McLaughlin:RogersRamanujanSlater}
James {Mc Laughlin}, Andrew~V. Sills, and Peter Zimmer,
  \emph{{Rogers-Ramanujan-Slater Type Identities}}, Electronic Journal of
  Combinatorics \textbf{15} (2008), no.~DS15, 1--59.

\bibitem[Mur11]{Murakami:IntroToVolumeConjecture}
Hitoshi Murakami, \emph{{An Introduction to the Volume Conjecture}},
  Interactions Between Hyperbolic Geometry, Quantum Topology and Number Theory
  (Contemporary Mathematics) (Providence, Rhode Island) (Abhijit Champanerkar,
  Oliver~T. Dasbach, Efstratia Kalfagianni, Ilya Kofman, Walter~D. Neumann, and
  Neal~W. Stoltzfus, eds.), Amer Mathematical Society, 2011, pp.~1--40.

\bibitem[MV94]{MasbaumVogel:3valentGraphs}
Gregor Masbaum and Pierre Vogel, \emph{{3-valent graphs and the Kauffman
  bracket}}, Pacific J. Math. \textbf{164} (1994), no.~2, 361--381.

\bibitem[Slo11]{Sloane:IntegerSequences}
Neil J.~A. Sloane, \emph{{The On-Line Encyclopedia of Integer Sequences,
  http://oeis.org}}, 2011.

\bibitem[Tur88]{Turaev:YangBaxter}
Vladimir~G. Turaev, \emph{{The Yang-Baxter equation and invariants of links}},
  Invent. Math. \textbf{92} (1988), 527--553.

\bibitem[Wil99]{Wilf:JacobiTripleProduct}
Herbert~S. Wilf, \emph{{The number-theoretic content of the Jacobi triple
  product identity}}, Sem. Loth. Combin. \textbf{42} (1999), no.~B42k, 4 pp.

\bibitem[Zag11]{Zagier:TalkAtWalterfest}
Don Zagier, \emph{{Arithmetic properties of 3-dimensional quantum invariants}},
  Talk at the Walterfest, 2011.

\end{thebibliography}
\bibliographystyle {amsalpha}

\end{document}